\documentclass[a4paper, 11pt]{article}
\usepackage{calc, amsmath, amsthm, a4, latexsym, amssymb, layout, subfigure, bm, xspace, url, bbm} %,color
\usepackage[normalem]{ulem} % for striking through using \sout
\usepackage[dvips]{graphicx}
\usepackage[usenames,dvipsnames,svgnames,table]{xcolor}

\newtheorem{thm}{Theorem}[section]
\newtheorem{lemma}[thm]{Lemma}
\newtheorem{corollary}[thm]{Corollary}
\newtheorem{prop}[thm]{Proposition}

\theoremstyle{definition}
\newtheorem{defn}[thm]{Definition}

\newtheorem{rem}[thm]{Remark}

\newcommand{\be}[1]{\begin{equation}\label{#1}}
\newcommand{\ee}{\end{equation}}
\newcommand{\ba}{\begin{array}}
\newcommand{\ea}{\end{array}}
\newcommand{\bal}{\begin{aligned}}
\newcommand{\eal}{\end{aligned}}

\newcommand{\R}{\mathbb{R}}

\newcommand{\N}{\mathbb{N}}
\newcommand{\Q}{\mathbb{Q}}
\newcommand{\Z}{\mathbb{Z}}
\newcommand{\E}{\mathbb{E}}
\newcommand{\p}{\mathbb{P}}

\newcommand{\calB}{\mathcal{B}}
\newcommand{\calC}{\mathcal{C}}

\newcommand{\calF}{\mathcal{F}}

\newcommand{\calM}{\mathcal{M}}

\newcommand{\cB}{\mathcal{B}}
\newcommand{\cC}{\mathcal{C}}

\newcommand{\cM}{\mathcal{M}}

\newcommand{\cS}{\mathcal{S}}

\newcommand{\1}{1\hspace{-0.098cm}\mathrm{l}} % Franks Version 
\newcommand{\<}[1]{\left< #1 \right>}

\newcommand{\la}{\lambda}

\newcommand{\eps}{\varepsilon}

\newcommand{\ra}{\rightarrow}
\newcommand{\ssup}[1] {{\scriptscriptstyle{({#1}})}}
\newcommand{\sse}[1] {{\scriptscriptstyle{[{#1}}]}} %%%% for small superscript mit eckigen Klammern

 \newcommand{\supp}{{\rm supp}}

\newcommand{\tem}{\mathrm{tem}}
\newcommand{\rap}{\mathrm{rap}}
\renewcommand{\rho}{\varrho}

\begin{document}

\begin{center}
{\Large \bf The infinite rate symbiotic branching model:}\\[1mm]
{\Large \bf from discrete to continuous space}\\[5mm]
\vspace{0.7cm}
\textsc{Matthias Hammer\footnote{Institut f\"ur Mathematik, Technische Universit\"at Berlin, Stra\ss e des 17. Juni 136, 10623 Berlin, Germany.} and
Marcel Ortgiese\footnote{
Department of Mathematical Sciences, University of Bath, Claverton Down, Bath, BA2 7AY,
United Kingdom. } 
} 
\\[0.8cm]
{\small \today} 
\end{center}

\vspace{0.3cm}

\begin{abstract}\noindent 
The symbiotic branching model describes a spatial population
consisting of two types
that are allowed to migrate in space and branch locally only 
if both types are present.
We continue our investigation of the large scale behaviour of the system started in \cite{BHO15}, 
where we showed that the continuum system converges after
diffusive rescaling. 
Inspired by a scaling property of the continuum model,
a series of earlier 
works initiated 
by Klenke and Mytnik~\cite{KM11b, KM11c} studied the model on a discrete space, but with infinite branching rate. 
In this paper, we bridge the gap between 
the two models by showing that by diffusively
rescaling this discrete space infinite rate model, we obtain the continuum model from \cite{BHO15}.
As an application of this convergence result, we show that if we start the infinite rate system from complementary Heaviside initial conditions, the initial ordering of types is preserved in the limit and 
that the interface between the types consists of a single point.
  \par\medskip

  \noindent\footnotesize
  \emph{2010 Mathematics Subject Classification}:
  Primary\, 60K35,  \ Secondary\, 60J80, 60H15. 
  \end{abstract}

\noindent{\slshape\bfseries Keywords.} Symbiotic branching model, mutually catalytic
branching, stepping stone model, rescaled interface, moment duality, Meyer-Zheng topology.

\section{Introduction}\label{intro}

\subsection{The symbiotic branching model and its interface}

In \cite{EF04} Etheridge and Fleischmann introduce a spatial population model that 
describes the evolution of two interacting types. 
On the level of a  particle approximation,  the dynamics follows locally a branching process,
where each type branches with a rate proportional to the frequency
of the other type. Additionally, types are allowed to migrate to neighbouring 
colonies. In the 
continuum space and large population limit, the rescaled numbers 
of the respective types  $u^\sse{\gamma}_t(x)$ and $v^\sse{\gamma}_t(x)$ are given by the 
nonnegative solutions of
the system of stochastic partial differential equations
\begin{equation}\bal
\label{eqn:spde}
 	{\mathrm{cSBM}(\varrho,\gamma)}_{u_0, v_0}: \quad
 		\begin{cases}
  			\frac{\partial }{\partial t}
u_t^\sse{\gamma}(x) & \!\!\!\!= \frac{\Delta}{2} u^\sse{\gamma}_t(x) + 
                     	\sqrt{ \gamma u^\sse{\gamma}_t(x) v^\sse{\gamma}_t(x)} \, \dot{W}^\ssup{1}_t(x),\\[0.3cm]
 			\frac{\partial }{\partial t}v_t^\sse{\gamma}(x) 
 & \!\!\!\!= \frac{\Delta}{2} v_t^\sse{\gamma}(x) + 
                     	\sqrt{ \gamma u^\sse{\gamma}_t(x) v^\sse{\gamma}_t(x)} \, \dot{W}^\ssup{2}_t(x),%\\[0.3cm]
 		\end{cases}
	\eal\end{equation}
with suitable nonnegative initial conditions $u_0(x) \ge 0, v_0(x) \ge 0, x \in \R$.
Here, $\gamma > 0$ is the branching rate and $ (\dot{W}^\ssup{1},
\dot{W}^\ssup{2})$ is a pair of correlated standard Gaussian white noises on $\mathbb{R}_+ \times \mathbb{R}$ with correlation 
governed by a parameter $\varrho
\in [-1,1]$. 
Existence (for $\rho \in [-1,1]$) and uniqueness (for $\rho \in [-1,1)$) was
proved in~\cite{EF04} for a large class of initial conditions.

The model generalizes several well-known examples of spatial population dynamics.
Indeed, for $\varrho=-1$ and $u_0 = 1- v_0$,
the system reduces to the continuous-space {\em stepping stone model} analysed in \cite{T95}, while
for $\varrho=0$, 
the system is known as the {\em mutually catalytic model} due to Dawson and Perkins \cite{DP98}. 
Finally, for $\rho = 1$ and the extra assumption
$u_0= v_0$, 
the model is an instance of the \emph{parabolic Anderson model}, see for example~\cite{MuellerSupport91}.

One of the central question is
how the local dynamics, where one type will eventually dominate over the
other, interacts with the migration to shape the global picture. 
A particularly interesting situation is when initially both types are spatially
separated and one would like to know how one type `invades' the other, in other words
we would like to understand the interface between the two types.
Mathematically, this corresponds to
`complementary Heaviside initial conditions', i.e.
	\[
 	u_0(x) = {\bf 1}_{\R^-}(x) \quad \mbox{ and } \quad v_0(x) = {\bf 1}_{\R^+}(x), \quad x \in \R.
	\]

\begin{defn}\label{def:ifc}
		The interface at time $t$ of a solution $(u^\sse{\gamma}_t,v^\sse{\gamma}_t)_{t\ge0}$ of the symbiotic branching model $\mathrm{cSBM}(\varrho,\gamma)_{u_0, v_0}$ with $\varrho \in [-1,1]$, $\gamma>0$ is defined as 
		\begin{eqnarray*}
			{\rm Ifc}_t = \mbox{cl} \big\{x\in\R : u^\sse{\gamma}_t(x) v^\sse{\gamma}_t(x) > 0 \big\},
		\end{eqnarray*}
		where $\mbox{cl}(A)$ denotes the closure of the set $A$ in $\R$.
\end{defn}

The first question that arises is whether this interface is non-trivial. Indeed, in ~\cite{EF04} it is shown that the interface is a compact set
and moreover that the width of the interface growths at most linearly in $t$.
This result is strengthened in~\cite[Thm.\ 2.11]{BDE11} for all $\rho$ close to $-1$
by showing that the width is at most of order $\sqrt{t\log(t)}$.

Especially the latter bound on the width seems to suggest diffusive behaviour for the interface.
This conjecture is supported by
the following \emph{scaling property}, see  \cite[Lemma~8]{EF04}:
If we rescale time and space diffusively, i.e.\
if given $K > 0$ we set
\[ (u^\ssup{K}_t(x), v^\ssup{K}_t(x)\big) := \big(u^\sse{\gamma}_{K^2t}(K x), v^\sse{\gamma}_{K^2t}(K x)\big) \quad \mbox{for } x \in \R, t \geq 0, \]
then 
this defines  a solution to $\mathrm{cSBM}(\rho, K \gamma)_{u^\ssup{K}_0, v^\ssup{K}_0}$ with 
correspondingly transformed initial  states $(u^\ssup{K}_0, v^\ssup{K}_0)$.

Provided that the initial conditions are invariant under diffusive rescaling, 
 a diffusive rescaling of the system is equivalent (in law) to a rescaling of just the branching rate.
Since the complementary Heaviside initial conditions are invariant, 
we will in the following always consider the limit $\gamma \ra \infty$. 
This scaling then includes the diffusive rescaling, while also giving us the flexibility 
to consider more general initial conditions.

For the continuous space model this programme has been carried out in~\cite{BHO15}.
We define the  measure-valued processes
\begin{equation}\label{defn:mu_nu} \mu_t^\sse{\gamma} (dx)  := u^\sse{\gamma}_t(x)\, dx ,\quad  \quad \nu_t^\sse{\gamma} (dx)
:= v^\sse{\gamma}_t(x) \, dx
\end{equation}
obtained by taking the  solutions of $\mathrm{cSBM}(\rho,\gamma)_{u_0,v_0}$ as densities, where the initial conditions 
remain fixed. 
The following result was proved in \cite[Thm.\ 1.10]{BHO15}. 
Here and in the following, if $\cS=\R$ or $\cS=\Z^d$ with $d\in\N$ we denote by $\cM_\tem(\cS)$ the space of tempered measures on $\cS$, and by
$\cM_\rap(\cS)$ the space of rapidly decreasing measures. Similarly, 
$\cB^+_\tem(\cS)$ (resp.\ $\cB^+_\rap(\cS)$) denotes the space of nonnegative, tempered (resp.\ rapidly decreasing)
measurable functions on $\cS$. 
We collect all the relevant formal definitions in Appendix~\ref{appendix0}.

\begin{thm}[\cite{BHO15}]\label{thm:old2_new}
Let  $\rho\in(-1,0)$.
Suppose the initial conditions satisfy
$(u_0,v_0)\in\calB_\tem^+(\R)^2$ 
resp.\ $(u_0,v_0)\in\calB_\rap^+(\R)^2$,
and for each $\gamma >0$ we let $(u_t^\sse{\gamma},v_t^\sse{\gamma})_{t \geq 0}$ be the solution to $\mathrm{cSBM}(\rho,\gamma)_{u_0,v_0}$.
Then as $\gamma \ra \infty$, the measure-valued process $(\mu_t^\sse{\gamma},\nu_t^\sse{\gamma})_{t\ge0}$ defined by~\eqref{defn:mu_nu}
converges in law in $D_{[0,\infty)}(\calM_\tem(\R)^2)$ resp.\ in $D_{[0,\infty)}(\calM_\rap(\R)^2)$ equipped with the Meyer-Zheng `pseudo-path' topology to 
a measure-valued process $(\mu_t,\nu_t)_{t \geq 0}$ satisfying the following separation-of-types condition:
For any $x \in \R$, $t \in (0,\infty)$ we have
\begin{equation}\label{eq:sep_types} 
\E_{\mu_0, \nu_0} [ S_\eps \mu_t(x) S_\eps \nu_t(x)]  \ra 0 , \quad \mbox{as } \eps \rightarrow 0,
\end{equation}
where $(S_t)_{t \geq 0}$ denotes the heat semigroup.
\end{thm}

\begin{rem}
\begin{itemize}
 \item[(a)] We call the limit $(\mu_t,\nu_t)_{t \geq 0}$ the \emph{continuous-space
infinite rate symbiotic branching model} $\mathrm{cSBM}(\rho,\infty)_{u_0,v_0}$.
 \item[(b)] We recall the definition of the Meyer-Zheng `pseudo-path' topology in Appendix~\ref{sec:appendix_path}.
 This topology is strictly weaker than the standard Skorokhod topology on $D_{[0,\infty)}$. 
 Under the more restrictive condition that $(u_0,v_0) = (\1_{\R^-},\1_{\R^+})$ and $\rho \in (-1, -\frac{1}{\sqrt{2}})$, 
 we can also show tightness in the stronger Skorokhod topology, so that then in particular $(\mu^\sse{\gamma},\nu^\sse{\gamma})$
 converges in $\calC_{[0,\infty)}( \calM_\tem(\R)^2)$ as $\gamma \ra \infty$, cf.\ Theorem 1.5 in~\cite{BHO15}. Also, we show that 
 in this case, the limiting measures $\mu_t,\nu_t$ are absolutely continuous with respect to Lebesgue measure
 and if we denote the densities also by $\mu_t$ and $\nu_t$, we can derive the more intuitive separation-of-types condition:
 \begin{equation}\label{singularity1}
\mu_t(\cdot) \nu_t(\cdot)=0\qquad \p\otimes\ell\text{-a.s.}
\end{equation}
\end{itemize}
\end{rem}

For $\rho = -1$ and complementary Heaviside initial conditions,
the analogue of Theorem~\ref{thm:old2_new} was already proved in Tribe~\cite{T95} for the continuum stepping stone model, 
as one of the steps of understanding the diffusively rescaled interface. 
Under these assumptions it was shown that the process $(\mu_t^\sse{\gamma}, \nu_t^\sse{\gamma})_{t \geq 0}$
converges weakly for $\gamma \ra \infty$ to 
\begin{equation}\label{eq:tribe} (\1_{\{x\le B_t\}}\, dx,\1_{\{x\ge B_t\}}\, dx)_{t \geq 0} , \end{equation}
for $(B_t)_{t \geq 0}$ a standard Brownian motion. 
Unfortunately, our previous work does not give such a truly explicit characterization of the infinite rate system
for $\rho > -1$. However, we do have a characterization in terms of a martingale problem (which we will recall below).
This allows us to  show that the limit is not of the form~\eqref{eq:tribe}, see Remark~1.14 in~\cite{BHO15},
 even if we allow the position to be
a general diffusion rather than a Brownian motion.
 In fact, even the case $\rho = -1$ with general initial conditions is not covered
by~\cite{T95}. 
However, this case is taken up in the work \cite{HOV15}, where we show in particular that for complementary initial conditions which do not necessarily sum up to one, the interface of the infinite rate limit moves like a Brownian motion with drift.

For $\rho>-1$, in order to take a first step towards a more explicit characterization of the limit in Theorem~\ref{thm:old2_new}, 
our aim in this paper is to make the connection to related results on the discrete lattice $\Z$. 
We first recall that for any $d\in\N$, the \emph{discrete-space} finite rate symbiotic branching model on $\Z^d$
is given by the nonnegative solutions $((u^\sse{\gamma}_t(x), v^\sse{\gamma}_t(x)), x \in \Z^d, t \geq 0)$ of 
\begin{equation}\bal
\label{eqn:sde_discrete}
 	{\mathrm{dSBM}(\varrho,\gamma)}_{u_0, v_0}: \quad
 		\begin{cases}
%  			\frac{\partial }{\partial t}
du^\sse{\gamma}_t(x) & \!\!\!\!= \frac{\Delta}{2^d} u^\sse{\gamma}_t(x) \, dt + 
                     	\sqrt{ \gamma u^\sse{\gamma}_t(x) v^\sse{\gamma}_t(x)} \, {dW}^\ssup{1}_t(x),\\[0.3cm]
% 			\frac{\partial }{\partial t}
dv^\sse{\gamma}_t(x) 
 & \!\!\!\!= \frac{\Delta}{2^d}  v^\sse{\gamma}_t(x) \, dt+ 
                     	\sqrt{ \gamma u^\sse{\gamma}_t(x) v^\sse{\gamma}_t(x)} \, {dW}^\ssup{2}_t(x),%\\[0.3cm]
 		\end{cases}
	\eal\end{equation}
with suitable nonnegative initial conditions $u_0(x) \ge 0, v_0(x) \ge 0, x \in \Z^d$.
Here, $\gamma > 0$ is the branching rate, $\Delta$ is the discrete Laplace operator, defined for any $f : \Z^d\ra\R$ as
\begin{equation}\label{def:discreteLaplace} \Delta f(x) := \sum_{y:|y-x|=1}(f(y)-f(x)), 
\quad x\in \Z^d, \end{equation}
and the pair $(W^1(x), W^2(x))$ is a
$\varrho$-correlated two-dimensional Brownian motion
which is independent for each $x \in \Z^d$.

Prior to our work, but also inspired by the scaling property
for the continuous model, 
Klenke and Mytnik consider this discrete space model, where the branching rate is sent to infinity.
Indeed,
in a series of papers~\cite{KM10a, KM11b, KM11c} show that 
a non-trivial limiting process exists for $\gamma \to \infty$ (on the lattice) and  study its
long-term properties. Moreover, Klenke and Oeler \cite{KO10} give a Trotter type approximation.
Their results concentrate on the case $\rho = 0$, i.e.\ the mutually catalytic model, 
however analogous results have been derived by D\"oring and Mytnik for the case
$\varrho \in(-1, 1)$ in~\cite{DM11a, DM12}.
In analogy with \eqref{singularity1}, the limiting process satisfies the separation-of-types property, i.e.\ at each site only one type is present almost surely.
We will refer to the limit as
the \emph{discrete-space infinite rate symbiotic branching model}, abbreviated as $\mathrm{dSBM}(\rho,\infty)_{u_0,v_0}$. 

What makes the results on the lattice especially interesting for our purpose of identifying 
the continuous infinite rate model is the fact that there is a
very explicit description of the limit $\mathrm{dSBM}(\rho,\infty)$
in terms of an infinite system of jump-type stochastic differential equations (SDEs). 

As noted in~\cite{EF04}, 
the continuous finite rate symbiotic branching model $\mathrm{cSBM}(\rho,\gamma)$ can be obtained as a diffusive time/space
rescaling of the discrete model $\mathrm{dSBM}(\rho,\gamma)$. 
Therefore, 
it seems natural to expect that by rescaling the discrete system with infinite branching
rate
diffusively we obtain the infinite rate continuous space system of Theorem~\ref{thm:old2_new}.
In other words, we expect that the following diagram (Figure~\ref{fig:commute}) commutes.
\begin{figure}[htbp]
\centering
\includegraphics[width=0.75\textwidth]{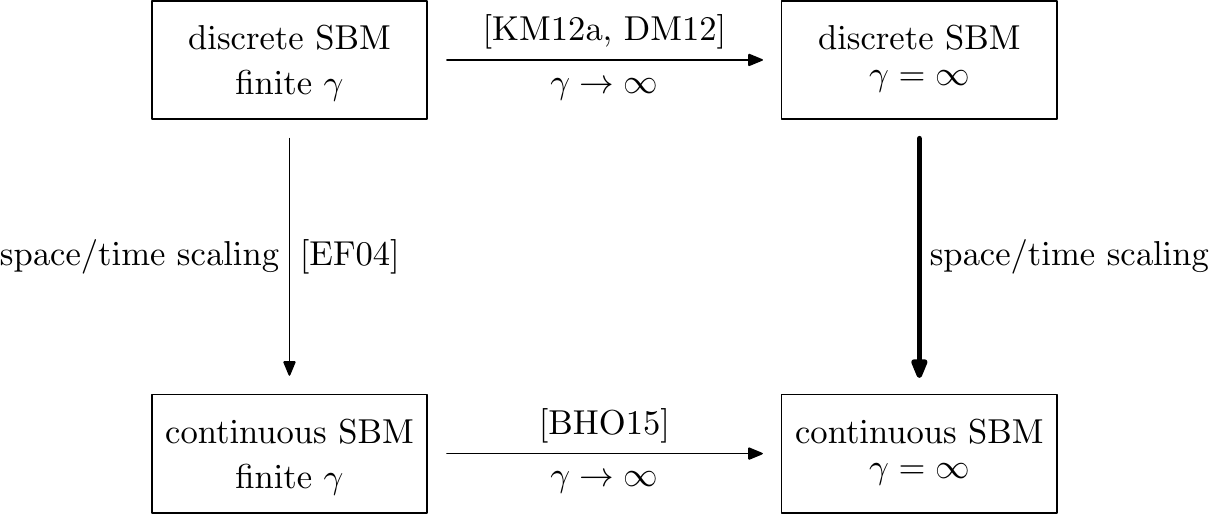}% 
\caption{A commuting diagram.}\label{fig:commute}
\end{figure}

Indeed, this will be our first main result in this paper. 
In future work, we will attempt
to exploit this commutativity to give a more explicit description of
the limiting object in Theorem~\ref{thm:old2_new} by
rescaling the jump-type SDEs of~\cite{KM11b}.
As the second main result in this paper, we can deduce from the scaling limit
that the continuous model preserves the initial ordering of types in the limit and
also that the interface consists of a single point.

%\newpage
\section{Main results} 

In order to state our main result, we first recall the martingale
problem that characterizes the limit in Theorem~\ref{thm:old2_new}.
This martingale problem is very much related to the martingale problem
for the discrete space model $\mathrm{dSBM}(\rho,\infty)$ of Klenke and Mytnik~\cite{KM11b}.

Throughout, we use the notation defined in Appendix~\ref{appendix0}. 
We can formulate the martingale problem in both discrete and continuous space simultaneously.
Therefore, let $\cS$ be either $\Z^d$ or $\R$.
We recall the self-duality function employed in \cite{EF04}:
Let $\rho\in(-1,1)$ and if either $(\mu,\nu,\phi,\psi)\in\calM_\tem(\cS)^2\times\calB_\rap(\cS)^2$ or $(\mu,\nu,\phi,\psi)\in\calM_\rap(\cS)^2\times\calB_\tem(\cS)^2$, 
define
\begin{equation}\label{self-duality product}
\bal \langle\langle \mu , \nu, \phi, \psi \rangle \rangle_\rho &:= -\sqrt{1-\rho}\,\langle \mu + \nu, \phi + \psi\rangle_\cS
+ i \sqrt{1+\rho}\, \langle \mu - \nu, \phi - \psi \rangle_\cS ,
\eal\end{equation}
where $\langle \mu , \phi \rangle_\cS$ denotes the integral $\int_\cS \phi(x)\, \mu(dx)$, for $\mu$ a measure and $\phi$ a measurable function. 
Then, we define the \emph{self-duality function $F$} as
\begin{equation}\label{self-duality function}
 F(\mu,\nu,\phi,\psi) := \exp \langle \langle \mu, \nu, \phi, \psi \rangle \rangle_\rho. 
\end{equation}
With this notation,  we define a martingale problem, which in the continuous setting was
called $\mathbf{MP}'$ in~\cite{BHO15}.

\begin{defn}[Martingale Problem $(\mathbf{MP}_F(\cS))_{\mu_0,\nu_0}^\rho$]\label{defn:MP'}
Fix $\rho\in(-1,1)$ and (possibly random) initial conditions $(\mu_0,\nu_0)\in\calM_\tem(\cS)^2$ (resp.\ $\calM_\rap(\cS)^2$).
A c\`adl\`ag $\calM_\tem(\cS)^2$-valued (resp.\ $\calM_\rap(\cS)^2$-valued) stochastic process $(\mu_t, \nu_t)_{t\ge 0}$
is called a solution to the martingale problem $(\mathbf{MP}_F(\cS))_{\mu_0,\nu_0}^\rho$
if the following holds: There exists an increasing c\`adl\`ag $\calM_\tem(\cS)$-valued (resp.\ $\calM_\rap(\cS)$-valued) process $(\Lambda_t)_{t\ge0}$ with $\Lambda_0=0$ and
\begin{equation}\label{finiteness Lambda 1}
\E_{\mu_0,\nu_0}\big[\Lambda_t(dx)\big]\in\calM_\tem(\cS)\quad(\text{resp.\ } \E_{\mu_0,\nu_0}\big[\Lambda_t(dx)\big]\in\calM_\rap(\cS))
\end{equation}
for all $t>0$, 
such that for all test functions
 $\phi,\psi\in\big(\calC_{\rap}^{(2)}(\cS)\big)^+$ (resp.\ $\phi,\psi\in\big(\calC^{(2)}_{\tem}(\cS)\big)^+$) 
the process
\begin{align}\label{MP9}\bal
 F(\mu_t, \nu_t,\phi,\psi) &- F(\mu_0,\nu_0,\phi,\psi) \\ &- \frac{1}{2}\int_0^t F(\mu_s, \nu_s,\phi,\psi)\,\langle\langle \mu_s, \nu_s, \Delta\phi, \Delta\psi\rangle\rangle_\rho \, ds \\
& - 4(1-\rho^2)\int_{[0,t]\times\cS} F(\mu_s,\nu_s,\phi,\psi)\,\phi(x)\psi(x)\,\Lambda(ds,dx)
\eal\end{align}
is a martingale, where
$\Delta$ denotes the continuum Laplace operator if $\cS=\R$ and the discrete Laplace operator if $\cS=\Z^d$. 
\end{defn}

In \eqref{MP9} we have interpreted the right-continuous and increasing
process $t\mapsto\Lambda_t(dx)$ as a (locally finite) measure $\Lambda(ds,dx)$ on $\R^+\times\cS$, via
\[\Lambda([0,t]\times B):=\Lambda_t(B).\]

In order to characterize $\mathrm{cSBM}(\rho,\infty)$, 
it does not suffice to require that the martingale problem 
$(\mathbf{MP}_F(\R))_{u_0,v_0}^\rho$ is satisfied, since it
holds for $\mathrm{cSBM}(\rho,\gamma)$ for arbitrary $\gamma<\infty$, see e.g.\ Proposition~A.5 in~\cite{BHO15}.
However, we do get uniqueness if we require additionally that 
the separation-of-types condition \eqref{eq:sep_types} is satisfied, as we recall from~\cite[Thm.\ 1.10]{BHO15}
(where the martingale problem $\mathbf{MP}_F(\R)$ was denoted by $\mathbf{MP}'$).

We note that in the discrete context, our martingale problem $(\mathbf{MP}_F(\Z^d))^\rho$ is not exactly the same
as the martingale problem in~\cite[Thm.~1.1]{KM11b}. Indeed, the main 
difference is the appearance of the measure $\Lambda$,
which, in some sense that can be made precise, characterizes the correlations. 
The reason why we need this extra term in the continuous case
can be understood if we recall that the martingale problem $\mathbf{MP}_F$
is tailored to an application of a self-duality (introduced in this context by Mytnik~\cite{Mytnik98}), 
which characterizes the finite-dimensional distributions.
In the discrete context it suffices to consider test functions $\phi, \psi$ 
that satisfy $\phi(x)\psi(x) = 0$ for all $x \in \Z^d$, see Corollary 2.4 in~\cite{KM10a}.
However, the same arguments do not carry over to the continuous space, where 
we need arbitrary test functions $\phi,\psi\in\big(\calC_{\rap}^{(2)}(\R)\big)^+$ (resp.\ $\phi,\psi\in\big(\calC^{(2)}_{\tem}(\R)\big)^+$).

But we note that obviously any solution of our martingale problem $\mathbf{MP}_F(\Z^d)$
(together with separation-of-types) 
satisfies the martingale problem of Theorem 1.1 in~\cite{KM11b} (respectively Theorem~4.4 in~\cite{DM12}
for general $\rho$).
So as a first preliminary result, we show that the converse is also true
and that there is a unique solution to the discrete analogue of the martingale problem in~\cite{BHO15}. 
Moreover, we allow for more general initial conditions. 
We combine the existence and uniqueness result for both the discrete and the continuous case in the following theorem, 
where for a measure $\nu$ on $\Z^d$ we write $\nu(k)$ instead of $\nu(\{k \})$.

\begin{thm}\label{thm:MPinf_discrete}
Assume that $\rho\in(-1,0)$. Consider $\cS\in\{\Z^d,\R\}$.
\begin{enumerate}
\item[a)] 
For all initial conditions $(\mu_0,\nu_0)\in\calM_\tem(\cS)^2$ (resp.\ $(u_0,v_0)\in\calM_\rap(\cS)^2$), there exists a unique solution $(\mu_t,\nu_t)_{t\ge0}$ to the martingale problem $(\mathbf{MP}_F(\cS))_{\mu_0,\nu_0}^\rho$
satisfying the separation-of-types property in the sense that 
\begin{itemize}
\item if $\cS=\Z^d$, then for all $t>0$ and $k\in\Z^d$ we have
\begin{equation}\label{eq:dsep}
\mu_t(k)\nu_t(k)=0\qquad \p_{\mu_0,\nu_0}\text{-a.s.};
\end{equation}
\item if $\cS=\R$, then for all $t>0$ and $x\in\R$ we have
\begin{equation}\label{eq:csep} 
S_{t+\eps}\mu_0(x) S_{t+\eps}\nu_0(x)  \geq \E_{\mu_0, \nu_0} [ S_\eps \mu_t(x) S_\eps \nu_t(x)]  \xrightarrow{\eps \to 0} 0 .
\end{equation}
\end{itemize}
Moreover, the solution is a strong Markov process.
\item[b)] Let $(u_0,v_0)\in\calB^+_\tem(\cS)^2$ (resp.\ $(u_0,v_0)\in\calB^+_\rap(\cS)^2$). For each $\gamma>0$ denote by $(u_t^\sse{\gamma},v_t^\sse{\gamma})_{t\ge0}$ the solution to $\mathrm{SBM}(\rho,\gamma)_{u_0,v_0}$, 
considered as measure-valued processes. Then, as $\gamma\uparrow\infty$, the processes $(u^\sse{\gamma}_t,v^\sse{\gamma}_t)_{t\ge0}$ converge in law in $D_{[0,\infty)}(\calM_\tem(\cS)^2)$ 
(resp.\ in $D_{[0,\infty)}(\calM_\rap(\cS)^2)$) 
equipped with the Meyer-Zheng ``pseudo-path" topology
to the unique solution of the martingale problem $(\mathbf{MP}_F(\cS))_{u_0,v_0}^\rho$
satisfying the separation-of-types condition.
\end{enumerate}
\end{thm}
We call the unique solution to the martingale problem $(\mathbf{MP}_F(\cS))^\rho$ satisfying \eqref{eq:dsep} resp.~\eqref{eq:csep}
the \emph{infinite rate symbiotic branching process} and denote it by $\mathrm{dSBM}(\rho,\infty)$ if $\cS=\Z^d$ and by $\mathrm{cSBM}(\rho,\infty)$ if $\cS=\R$. 

\begin{rem}
As noted above, for the discrete case our martingale problem is more restrictive than the version of~\cite{KM11b,DM12}, since we 
require the martingale property to hold for a larger class of test functions. 
Thus, for $\cS=\Z^d$ our theorem generalizes their results in two ways: We show that their solution also satisfies our stronger martingale problem.
Further, we allow for more general initial conditions since we do not require the types to be separated initially, while~\cite{KM11b,DM12} 
assume that $\mu_0(k)\nu_0(k) = 0$ for all $k \in \Z^d$.
Under this condition, by uniqueness our solution coincides of course with the
infinite rate process constructed in \cite{KM11b} and \cite{DM12}.

Nevertheless, the work in~\cite{KM11b} goes substantially beyond what we claim here in the sense that they are also able
to show that the solution of $\mathrm{dSBM}(\rho,\infty)$ can be characterized as a solution to a jump-type
SDE, see~\cite[Thm~1.3]{KM11b} for $\rho =0$ and~\cite[Prop.~4.14]{DM12} for $\rho \ne 0$.
Moreover,~\cite{KM11b} considers more general operators than the discrete Laplacian. 
Also, they define solutions as taking values in a  Liggett-Spitzer space (characterized by a suitable test function 
$\beta : \Z^d \ra \R^+$), whereas we follow~\cite{DP98} in using tempered measures as state space. 
By choosing $\beta$ in a suitable way, one can show that for initial conditions
that satisfy~\eqref{eq:dsep} our solution agrees with theirs.
\end{rem}

Now we can finally state the main result of our paper, which says that for $\rho\in(-1,0)$ the (one-dimensional) discrete-space infinite rate model $\mathrm{dSBM}(\rho,\infty)$ 
converges under diffusive rescaling (in the Meyer-Zheng sense) to the continuous-space model $\mathrm{cSBM}(\rho,\infty)$ introduced in~\cite{BHO15}. 
More precisely, given initial conditions $(\mu_0,\nu_0)$ for $\mathrm{cSBM}(\rho,\infty)$
we define for each $n\in\N$ initial conditions $(u_0^\ssup{n},v_0^\ssup{n})$ for $\mathrm{dSBM}(\rho,\infty)$ by
\begin{equation}\label{defn:variabel initial conditions}
u_0^\ssup{n}(k):=n\,\mu_0\left([\tfrac{k}{n},\tfrac{k+1}{n})\right) \quad\text{and}\quad v_0^\ssup{n}(k) := n\,\nu_0\left([\tfrac{k}{n},\tfrac{k+1}{n})\right),\qquad k\in\Z.
\end{equation}
It is easy to see that for $(\mu_0,\nu_0)\in\cM_\tem(\R)^2$ (resp.\ $\cM_\rap(\R)^2$), we have $(u_0^\ssup{n},v_0^\ssup{n})\in\cM_\tem(\Z)^2$ (resp.\ $\cM_\rap(\Z)^2$).
Denote by $(u_t^\ssup{n},v_t^\ssup{n})_{t\ge0}$ the solution to $\mathrm{dSBM}(\rho,\infty)_{u_0^\ssup{n},v_0^\ssup{n}}$. 
We define a sequence $(\mu_t^\ssup{n},\nu_t^\ssup{n})_{t\ge0}$ of approximating processes for $\mathrm{cSBM}(\rho,\infty)_{\mu_0,\nu_0}$ by diffusive rescaling, as follows:
For any Borel subset $B\subseteq\R$, let
\begin{equation}\label{defn:rescaling}
\mu^\ssup{n}_t(B) := \frac{1}{n} \sum_{k\in\Z}u^\ssup{n}_{n^2 t}(k)\, \1_{B}(k/n) \quad\text{and} \quad\nu^\ssup{n}_t(B) := \frac{1}{n} \sum_{k\in\Z}v^\ssup{n}_{n^2 t}(k)\, \1_{B}(k/n) , \qquad  t \geq 0 . 
\end{equation}
Observe that for each $n\in \N$, the measures $\mu^\ssup{n}_t$
are concentrated on the scaled lattice $\frac{1}{n}\Z$, with $\frac{1}{n}u^\ssup{n}_{n^2t}(n\,\cdot)$ as density w.r.t.\ counting measure, 
and analogously for $\nu^\ssup{n}_t$. 
Considered as discrete measures on $\R$, it is easy to see that indeed $(\mu_t^\ssup{n},\nu_t^\ssup{n})_{t\ge0}$ takes values in $D_{[0,\infty)}(\cM_\tem(\R)^2)$ (resp.\ $D_{[0,\infty)}(\cM_\rap(\R)^2)$).
Also note that
from \eqref{defn:variabel initial conditions}, 
it is clear that 
\begin{equation}\label{conv variabel initial conditions}
(\mu_0^\ssup{n}, \nu_0^\ssup{n})\to(\mu_0,\nu_0)
\end{equation}
in $\calM_\tem(\R)^2$ (resp.\ $\cM_\rap(\R)^2$) as $n\to\infty$.

\begin{thm}\label{thm:conv}
Let $\rho\in(-1,0)$ and consider 
initial conditions $(\mu_0,\nu_0)\in\calM_\tem(\R)^2$ (resp.\ $\calM_\rap(\R)^2$). 
For each $n\in\N$ denote by $(u_t^\ssup{n},v_t^\ssup{n})_{t \geq 0}$ the solution to $\mathrm{dSBM}(\rho,\infty)_{u_0^\ssup{n},v_0^\ssup{n}}$ from Theorem \ref{thm:MPinf_discrete} for $\cS=\Z$, 
with initial conditions $(u_0^\ssup{n},v_0^\ssup{n})$ defined by \eqref{defn:variabel initial conditions}.
Then as $n\to\infty$, the sequence of 
processes $(\mu_t^\ssup{n},\nu_t^\ssup{n})_{t \geq 0}$ from \eqref{defn:rescaling} converges weakly in $D_{[0,\infty)}(\calM_\tem(\R)^2)$ (resp.\ $D_{[0,\infty)}(\calM_\rap(\R)^2)$) equipped with the Meyer-Zheng `pseudo-path' topology to the unique solution $(\mu_t,\nu_t)_{t \geq 0}$ of $\mathrm{cSBM}(\rho,\infty)_{\mu_0,\nu_0}$ from 
Theorem~\ref{thm:MPinf_discrete} for $\cS=\R$.
\end{thm}

\begin{rem}\label{rem:initial conditions}
For the convergence result of Theorem \ref{thm:conv}, it is of course not essential that the initial conditions for the rescaled discrete model are given exactly as in \eqref{defn:variabel initial conditions}. For example, \eqref{defn:variabel initial conditions} may be replaced by 
\begin{equation}\label{defn:variabel initial conditions 2}
u_0^\ssup{n}(k):=n\,\mu_0\left((\tfrac{k}{n},\tfrac{k+1}{n}]\right) \quad\text{and}\quad v_0^\ssup{n}(k) := n\,\nu_0\left((\tfrac{k}{n},\tfrac{k+1}{n}]\right),\qquad k\in\Z.
\end{equation}
What we really need is that $(u_0^\ssup{n},v_0^\ssup{n})$ be defined in such a way that \eqref{conv variabel initial conditions} holds. 
\end{rem}

In order to state our next result, whose proof is an application of the convergence in Theorem \ref{thm:conv},
we need some more notation and definitions:
For a Radon measure $\mu$ on $(\R,\calB(\R))$, we denote by $\supp(\mu)$ its measure-theoretic support, i.e.
\[ \supp(\mu):=\{x\in\R :  \mu\left(B_\eps(x)\right)>0 \text{ for all }\eps>0\} . \]
Further, let
\[L(\mu):=\inf\supp(\mu)\in\bar\R,\qquad R(\mu):=\sup\supp(\mu)\in\bar\R\]
denote the leftmost resp. rightmost point in the support of $\mu$. 
Note that $\mu=0$ if and only if $\supp(\mu)=\emptyset$, which is equivalent to $L(\mu)=\infty$, $R(\mu)=-\infty$.
The measure $\mu$ is called \emph{strictly positive} if its support is the whole real line, or equivalently if it is non-zero on every non-empty open set.

\begin{thm}\label{thm:SST}
Suppose $\rho \in (-1,0)$,
and assume 
initial conditions $(\mu_0,\nu_0)\in\cM_\tem(\R)^2$ or $\cM_\rap(\R)^2$ which are mutually singular and such that $R(\mu_0)\le L(\nu_0)$. 
Assume further that $\mu_0+\nu_0$ is not the zero measure. 
Let $(\mu_t,\nu_t)_{t\ge0}$ denote the solution to $\mathrm{cSBM}(\rho,\infty)_{\mu_0,\nu_0}$. 
Then the following holds:
\begin{itemize}
 \item[a)] 
The process $(\mu_t,\nu_t)_{t \geq 0}$ preserves the \emph{initial ordering of types} in the sense that, almost surely,
\begin{equation}\label{eq:SST} R(\mu_t) \leq L(\nu_t) \quad \mbox{for all } t \geq 0. \end{equation}
\item[b)]
For all fixed $t > 0$, almost surely, the 
measures $\mu_t$ and $\nu_t$ are mutually singular and have
a \emph{single-point interface} in the sense that 
\[R(\mu_t) = L(\nu_t),\]
and the sum $\mu_t+\nu_t$ is strictly positive.
\end{itemize}
\end{thm}

\begin{rem}
Of course, Theorem \ref{thm:SST} holds in particular for complementary Heaviside initial conditions $(\mu_0,\nu_0)=(\1_{\R^-},\1_{\R^+})$ as considered in \cite{BHO15}.
Its proof 
proceeds by first showing the analogous result for the discrete-space model and then using the convergence result of Theorem \ref{thm:conv}, see Section~\ref{sec:proof SPI} below. 
Note that the property \eqref{eq:SST} holds 
pathwise on a set of probability one, while
the second part of the theorem is restricted to \emph{fixed} times and does not ensure existence of an `interface process' $(I_t)_{t\ge0}$ such that almost surely we have $I_t:= R(\mu_t)=L(\nu_t)$ for all $t>0$.
However, the restriction to fixed times applies also to the separation-of-types property \eqref{eq:csep} and to the absolute continuity in \cite[Thm.\ 1.5]{BHO15}. 
It is in fact notoriously difficult for this class of continuous-space models to obtain uniform-in-time results (see e.g. the discussion in \cite[Sec.~7]{Dawsonetal2002a} concerning the two-dimensional finite rate mutually catalytic branching model).
In our case, we deduce $R(\mu_t)  = L(\mu_t)$ by combining~\eqref{eq:SST} 
with strict positivity of the sum $\mu_t+\nu_t$,
and it is this latter property that we can only show for fixed times and not in a pathwise sense.
Similarly, for the mutual singularity of the measures we use~\eqref{eq:SST} together with the `separation of types'-condition \eqref{eq:csep}, and again the latter is a property for fixed times.
Finally, note that also for discrete space the `interface problem', i.e. the question whether the discrete analogue of the `single-point interface' property holds pathwise almost surely, is open so far, see e.g. \cite{KO10}, p. 485. 
In fact, it is not even known whether the process is strictly positive pathwise, see the conjecture on p. 11 in \cite{KM11b}.
\end{rem}

For the proof of our convergence result Thm.\ \ref{thm:conv},  
we need that in both discrete and continuous space, $\mathrm{SBM}(\rho,\infty)$ has $(2+\eps)$-th moments if $\rho<0$ and $\eps$ is chosen sufficiently small.
This is a generalization of \cite[Thm.\ 1.2]{DM11a}, who consider the discrete model. 
In fact, we can show that $\mathrm{SBM}(\rho,\infty)$ has finite $p$-th moments for any $p>2$ such that $\rho$ is `sufficiently close' to $-1$ w.r.t.\ $p$. 
Moreover, second moments can be calculated explicitly.
In order to formulate these results, it is convenient to introduce the following notation: 
Again let $\cS$ be either $\R$ or $\Z^d$ for some $d\in\N$, and let $(S_t)_{t\ge0}$ denote the usual heat semigroup on $\cS$, 
i.e. the semigroup of standard Brownian motion if $\cS=\R$ and the semigroup of simple symmetric (continuous-time) random walk if $\cS=\Z^d$.
Further, we write $(S_t^\ssup{2})_{t\ge0}$ for the corresponding two- resp.\ $2d$-dimensional semigroup on $\cS^2$.
Finally, we define a semigroup $(\tilde S_t)_{t\ge0}$ of the respective process killed upon hitting the diagonal in $\cS^2$, i.e.
\begin{equation}\label{eq: tilde S discrete}
{\tilde S}_tf(x,y) := \E_{x,y}\left[f(X_t^\ssup{1},X_t^\ssup{2})\1_{\{t<\tau^{1,2}\}}\right], \qquad f:\cS^2\to\R,\;(x,y)\in\cS^2,
\end{equation}
where $\tau^{1,2}:=\inf\{t>0: X_t^\ssup{1}=X_t^\ssup{2}\}$ denotes the first hitting time of the diagonal. 
Here  $(X^\ssup{1},X^\ssup{2})$ denotes a simple symmetric (continuous-time) $2d$-dimensional random walk if $\cS=\Z^d$ and a two-dimensional standard Brownian motion if $\cS=\R$.
We remark that $\tilde S_t$ is symmetric and 
in the continuous case 
has a transition density $\tilde p_t$ w.r.t.\ Lebesgue measure which can be expressed in terms of the usual heat kernel, see \eqref{transition density tilde S} in the appendix. 
Thus we can also let $\tilde S_t$ act on tempered or rapidly decreasing measures $m$ on $\R^2$ via
\[\tilde S_tm(x,y):=\iint_{\R^2}\tilde p_t(x,y;a,b)\,m(d(a,b)),\qquad (x,y)\in\R^2,\]
and from \eqref{transition density tilde S} it is easy to see that the function $\tilde S_tm$ is continous and vanishes on the diagonal in $\R^2$.

\begin{prop}[Moments of $\mathrm{SBM}(\rho,\infty)$]\label{prop:moments}
Assume $\rho\in(-1,0)$.
Consider initial conditions $(u_0,v_0)\in\calB^+_\tem(\cS)^2$ (resp.\ $(u_0,v_0)\in\calB^+_\rap(\cS)^2$), 
and let $(\mu_t,\nu_t)_{t\ge0}$ denote the solution of $\mathrm{SBM}(\rho,\infty)_{u_0,v_0}$. 
\begin{itemize}
\item[a)] 
We have the explicit second moment formulas 
\begin{equation}\label{second mixed moment 2}
\E_{u_0,v_0}\left[\langle \mu_t,\phi\rangle_\cS\,\langle \nu_t,\psi\rangle_\cS\right]=\big\langle \phi\otimes\psi, {\tilde S}_t(u_0\otimes v_0)\big\rangle_{\cS^2},
\end{equation}
\begin{align}\label{second moment 2}\bal
\E_{u_0,v_0}\left[\langle \mu_t,\phi\rangle_\cS^2\right]&=
\langle \phi, S_tu_0\rangle_\cS^2 + \frac{1}{|\rho|}\big\langle \phi\otimes\phi, (S_t^\ssup{2}-{\tilde S}_t)(u_0\otimes v_0)\big\rangle_{\cS^2},\\
\E_{u_0,v_0}\left[\langle \nu_t,\psi\rangle_\cS^2\right]&=
\langle \psi, S_tv_0\rangle_\cS^2 + \frac{1}{|\rho|}\big\langle \psi\otimes\psi, (S_t^\ssup{2}-{\tilde S}_t)(u_0\otimes v_0)\big\rangle_{\cS^2}
\eal\end{align}
for all $t>0$ and test functions $\phi,\psi\in\bigcup_{\lambda>0}\calC_{-\lambda}^+(\cS)$ (resp.\ $\bigcup_{\lambda>0}\calC_\lambda^+(\cS)$).
\item[b)]
Let $p>2$ such that $\rho+\cos(\pi/p)<0$. Then we have finiteness of $p$-th moments
\begin{equation}\label{finite p-moment}
\sup_{r\in[0,T]}\E_{u_0,v_0}\left[|\langle \mu_r,\phi\rangle_\cS|^{p}\right] < \infty,\qquad \sup_{r\in[0,T]}\E_{u_0,v_0}\left[|\langle \nu_r,\phi\rangle_\cS|^{p}\right] 
<\infty
\end{equation}
for all $T>0$ and test functions $\phi,\psi\in\bigcup_{\lambda>0}\calC_{-\lambda}^+(\cS)$ (resp.\ $\bigcup_{\lambda>0}\calC_\lambda^+(\cS)$).
\end{itemize}
\end{prop}

\begin{rem} 
\begin{itemize}
\item[a)]
In the discrete-space case $\cS = \Z^d$, and for integrable initial conditions with disjoint support, the  mixed second moment formula in~\eqref{second mixed moment 2}
is already known, see~\cite[Thm.\ 1.2]{DM11a}. Moreover, they state estimate \eqref{finite p-moment} only for the total masses, i.e. $\phi\equiv1$, and for $p=2+\eps$. 
The key observation for the proof of Proposition \ref{prop:moments} (also due to \cite{DM11a}) is that 
if $p>2$ and $\rho+\cos(\pi/p)<0$, then $p$-th moments of the finite rate processes $\mathrm{SBM}(\rho,\gamma)$ are bounded uniformly in $\gamma>0$, see Corollary~\ref{cor:uniform p-th moment 2} below. 
\item[b)] 
Generalizing \eqref{second mixed moment 2}-\eqref{second moment 2}, an explicit expression for the moments of $\mathrm{SBM}(\rho,\infty)$ is in fact available for all \emph{integer} values $p=n>2$ such that $\rho+\cos(\pi/n)<0$, which is however much more involved than for second moments. See \cite{HOV15} where this is proved by establishing a new moment duality for $\mathrm{SBM}(\rho,\infty)$. 
\end{itemize}
\end{rem}

\begin{rem}\label{rem:rho=-1}
The reader will have noticed that in all our results, we have omitted the boundary case $\rho=-1$. 
The reason is that in both discrete and continuous space, 
different techniques are required, since we can no longer use the self-duality with respect to the function $F$ from \eqref{self-duality function}
to show uniqueness. 
Instead, if $\rho = -1$ and the initial conditions satisfy $u_0 + u_0 = 1$, then  one can use the 
the moment duality with a system of coalescing random walks resp. Brownian motions, 
see~\cite{T95}. However,  if $u_0+v_0 \ne1$  a new approach is needed
(as remarked in~\cite{DM11a}).
This challenge is taken up in \cite{HOV15} where we construct $\mathrm{SBM}(-1,\infty)$ for general initial conditions, using a new moment duality instead of the self duality to establish uniqueness. 
Using these techniques, one can show that  all results in this section, in particular the convergence in Theorem \ref{thm:conv}, continue to hold for $\rho=-1$ as well.
\end{rem}

The remaining paper is structured as follows: 
In Section~\ref{sec:discrete} we prove
Theorem~\ref{thm:MPinf_discrete} 
and Proposition~\ref{prop:moments}.
Then, in Section~\ref{sec:convergence}, we show our main result
Theorem~\ref{thm:conv}.
Finally, as an application of the convergence result of Theorem~\ref{thm:conv}, in Section \ref{sec:proof SPI} we prove Proposition~\ref{thm:SST}.

{\bf Notation:} We have collected some of the standard facts and notations
about measure-valued processes in Appendix~\ref{appendix0}. In Appendix~\ref{sec:semigroup}, 
we recall some standard results for the (killed) heat semigroup and in Appendix~\ref{sec:appendix_path}
we recall the Meyer-Zheng ``pseudo-path'' topology.
Throughout this paper, we will denote by $c,C$ generic constants whose
value may change from line to line. If the dependence on parameters
is essential we will indicate this correspondingly.

{\bf Acknowledgments:} 
This project received financial support
by the German Research Foundation (DFG)  within
the DFG Priority Programme 1590 `Probabilistic Structures in
Evolution', grants no.\ BL 1105/4-1 and OR 310/1-1. The authors would like to thank Florian V\"ollering for many helpful discussions on the topic of this work.

\section{Existence, uniqueness and properties of $\mathbf{SBM}(\rho,\infty)$}\label{sec:discrete}

In this section, we sketch the proof of Theorem \ref{thm:MPinf_discrete}, i.e.\ 
existence and uniqueness (subject to the separation-of-types property) of the solution to the martingale problem $(\mathbf{MP}_F(\cS))_{u_0,v_0}^\rho$, 
for $\cS\in\{\Z^d,\R\}$, $\rho\in(-1,0)$ and \emph{general} initial conditions (not necessarily mutually singular). 
The solution is given as the $\gamma\uparrow\infty$-limit in the Meyer-Zheng topology of the finite rate processes $\mathrm{SBM}(\rho,\gamma)$. 
The continuous space version of Theorem \ref{thm:MPinf_discrete} (apart from the strong Markov property) was already proved in \cite{BHO15}, and the discrete space case is very similar. 
In order to make this note self-contained and for the convenience of the reader, we restate the main steps of the proof in Section \ref{ssec:proof}.
This reminder also offers some guidance for Section~\ref{sec:convergence} below, where
we use a similar general strategy for the proof of the convergence of the discrete to the continuous model.
In Section \ref{ssec:further_properties}, we prove some additional properties of 
$\mathrm{SBM}(\rho,\infty)$ (in particular the moment results in Proposition \ref{prop:moments}), 
which will be used in Section~\ref{sec:convergence} below, 
and the continuous-space versions of which were not contained in our earlier paper~\cite{BHO15}.
Therefore, in Section \ref{ssec:further_properties} we provide somewhat more detailed proofs.

We begin with some preliminaries on the finite rate model. 
Let $\cS\in\{\Z^d,\R\}$. 
For initial conditions $( u_0, v_0 ) \in \calB^+_\rap(\cS)^2$ resp.\ $\calB^+_\tem(\cS)^2$, we denote by 
$( u_t^\sse{\gamma}, v_t^\sse{\gamma})_{t\ge0}$ 
the solution to $\mathrm{SBM}(\varrho,\gamma)_{u_0,  v_0}$ in discrete or continuous space with these initial conditions and finite branching rate $\gamma\in(0,\infty)$. 
Considering the solutions as measure-valued processes, we have $( u_t^\sse{\gamma}, v_t^\sse{\gamma})_{t\ge0}\in\calC_{[0,\infty)}(\calM_\rap(\cS))^2$ resp.\ $\calC_{[0,\infty)}(\calM_\tem(\cS))^2$.
Further, we define a continuous $\calM_\rap(\cS)$- resp. $\calM_\tem(\cS)$-valued increasing process $(L_t^\sse{\gamma})_{t\ge0}$ by
\be{defn:L^gamma}
L_t^\sse{\gamma}(dx):=\gamma\int_0^t  u_s^\sse{\gamma}(x) v_s^\sse{\gamma}(x) \, ds\,dx,\qquad t\ge0,\ 
\ee
where $dx$ denotes counting measure if $\cS=\Z^d$ and Lebesgue measure if $\cS=\R$.
We will also consider $L^\sse{\gamma}$ as a measure $L^\sse{\gamma}(ds,dx)$ on $\R^+\times\cS$ via $L^\sse{\gamma}([0,t]\times B):=L^\sse{\gamma}_t(B)$.
Recalling that $(S_t)_{t\ge0}$ denotes the heat semigroup and $\Delta$ the Laplacian, 
by the martingale problem formulation of the SPDE \eqref{eqn:spde} and the system of SDEs \eqref{eqn:sde_discrete} we have that
for all test functions $\phi,\psi\in\bigcup_{\lambda>0}\calC^\ssup{2}_{-\lambda}(\cS)$ (resp.\ $\bigcup_{\lambda>0}\calC^\ssup{2}_\lambda(\cS)$) 
\begin{align}\label{eq:discrete MP 1_finite}\bal
M^\sse{\gamma}_t(\phi) &:= \langle u^\sse{\gamma}_t, \phi \rangle_\cS - \langle u_0, \phi \rangle_\cS - \frac{1}{2}\int_0^t\langle u^\sse{\gamma}_s, {\Delta}\phi \rangle_\cS\, ds , \\
N^\sse{\gamma}_t(\psi) &:= \langle v^\sse{\gamma}_t, \psi \rangle_\cS - \langle v_0, \psi \rangle_\cS - \frac{1}{2}\int_0^t\langle v^\sse{\gamma}_s, {\Delta}\psi \rangle_\cS\, ds 
\eal\end{align}
are square-integrable martingales with quadratic (co-)variation
\begin{align}\bal\label{eq:discrete MP 1a_finite} 
{}  [ M^\sse{\gamma}(\phi),  M^\sse{\gamma}(\phi) ]_t  = [ N^\sse{\gamma}(\phi), N^\sse{\gamma}(\phi) ]_t &= \langle L^\sse{\gamma}_t,\phi^2\rangle_\cS, 
  \\
 [  M^\sse{\gamma}(\phi) , N^\sse{\gamma}(\psi) ]_t  & = \rho \, \langle L^\sse{\gamma}_t,\phi\psi\rangle_\cS
\eal\end{align}
for all $t>0$, with $L_t^\sse{\gamma}$ from \eqref{defn:L^gamma}.
Also, we have the \emph{Green function representation} for $\mathrm{SBM}(\varrho,\gamma)_{ u_0,  v_0}$, 
see e.g.~\cite[Thm.\ 2.2(b)(ii)]{DP98} 
or~\cite[Cor.\ A.4]{BHO15}: 
For every $T>0$ 
and $\phi,\psi\in\bigcup_{\lambda>0}\calC_{-\lambda}(\cS)$ (resp.\ $\bigcup_{\lambda>0}\calC_{\lambda}(\cS)$) we know that 
\begin{align}\label{MP1a dual}\bal
 M^\sse{\gamma,T}_t(\phi) &:= \<{ u^\sse{\gamma}_t, S_{T-t}\phi}_{\cS} - \<{ u_0, S_T\phi}_{\cS},\qquad t\in[0,T],\\
 N^\sse{\gamma,T}_t(\phi) &:= \<{ v^\sse{\gamma}_t, S_{T-t}\phi}_{\cS} - \<{ v_0, S_T\phi}_{\cS},\qquad t\in[0,T]
\eal\end{align}
are martingales on $[0,T]$ with quadratic (co-)variation
\begin{align}\label{Cov1a dual}
  &[ M^\sse{\gamma,T}(\phi),  M^\sse{\gamma,T}(\phi) ]_t \notag\\ &
  \hspace{1cm} = [ N^\sse{\gamma,T}(\phi), N^\sse{\gamma,T}(\phi) ]_t 
  = \int_{[0,t]\times\cS} \left({S}_{T-r}\phi(x)\right)^2\, L^\sse{\gamma}(dr,dx), 
  \\
& [  M^\sse{\gamma,T}(\phi) , N^\sse{\gamma,T}(\psi) ]_t   = \rho  \int_{[0,t]\times\cS} {S}_{T-r}\phi(x)\, {S}_{T-r}\psi(x)\, L^\sse{\gamma}(dr,dx). \notag
\end{align}

\subsection{Proof of Theorem \ref{thm:MPinf_discrete}}\label{ssec:proof}

As in \cite{BHO15}, the first step in the proof of Theorem \ref{thm:MPinf_discrete} is to show tightness of the family of finite rate models $\mathrm{SBM}(\rho,\gamma)$, $\gamma\in(0,\infty)$.

\begin{prop}\label{tightness_MZ} 
Suppose $\rho\in[-1,0)$ and $( u_0,  v_0 ) \in \calB^+_\rap(\cS)^2$ (resp.\ $\calB^+_\tem(\cS)^2$). 
Then the family of processes $\{(u_t^\sse{\gamma},v_t^\sse{\gamma},L_t^\sse{\gamma})_{t\ge0}: \gamma>0\}$ is tight with respect to the Meyer-Zheng topology on $D_{[0,\infty)}(\calM_\rap(\cS)^3)$ (resp.\ $D_{[0,\infty)}(\calM_\tem(\cS)^3)$).
\end{prop}

The key step in the proof of the Meyer-Zheng tightness is the following lemma (corresponding to \cite[Lemma 3.1]{BHO15}) which relies crucially on the colored particle moment duality for finite rate symbiotic branching, 
see \cite[Prop.\ 9]{EF04} for the discrete case and \cite[Prop.\ 12]{EF04} for the continuous case.
The estimate shows that \eqref{Cov1a dual} is bounded in expectation, uniformly in $\gamma>0$. 
Recall that $(S^\ssup{2}_t)_{t\ge0}$ and $(\tilde S_t)_{t\ge0}$ denote the
heat semigroup on $\cS^2$ and the killed semigroup defined in \eqref{eq: tilde S discrete}, respectively.

\begin{lemma}\label{lemma boundedness quadratic variation}
Suppose 
$\rho\in[-1,0)$ and $( u_0,  v_0 ) \in \calB^+_\rap(\cS)^2$ (resp.\ $\calB^+_\tem(\cS)^2$). Then for all $t>0$ 
and test functions $\phi,\psi\in \bigcup_{\lambda>0}\calB_{-\lambda}^+(\cS)$ (resp.\ $\bigcup_{\lambda>0}\calB_\lambda^+(\cS)$)
we have monotone convergence as $\gamma\uparrow\infty$
\begin{align}\bal\label{Cov2 dual}
\E_{ u_0, v_0}\left[ \int_{[0,t]\times\cS} {S}_{t-r}\phi(x)\, {S}_{t-r}\psi(x)\, L^\sse{\gamma}(dr,dx)\right]
&\uparrow\frac{1}{|\rho|} \left\langle \phi\otimes\psi, ({S}_t^\ssup{2} - {\tilde S}_t)(u_0\otimes v_0)\right\rangle_{\cS^2}\\
&\le \frac{1}{ |\rho|}\langle \phi,{S}_tu_0\rangle_{\cS}\,\langle \psi, {S}_t v_0\rangle_{\cS}<\infty.
\eal\end{align}
\end{lemma}
For a proof in the continuous case $\cS=\R$, see \cite[Lemma 3.1]{BHO15}. 
The proof for the discrete case $\cS=\Z^d$ is virtually identical, replacing the Brownian motions by simple symmetric random walks and the corresponding local times. 
Note that in view of the definition of the semigroup $(\tilde S_t)_{t\ge0}$, the limit in \eqref{Cov2 dual} coincides indeed with $(35)$ in \cite{BHO15}. 

We now give a brief sketch of the proof of Proposition~\ref{tightness_MZ}.
In a different but similar setting, in Section~\ref{sec:convergence}, we will carry out the full details.

\begin{proof}[Sketch of the proof of Proposition~\ref{tightness_MZ}]
First use the Green function representation \eqref{MP1a dual}-\eqref{Cov1a dual} combined with the lower bound from \eqref{eq:uniform preservation} (for $n=1$), the Burkholder-Davis-Gundy inequality and the upper bound \eqref{Cov2 dual} to derive uniform moment estimates
\begin{equation}\label{uniform first moments dual process 1} 
\sup_{\gamma>0} \E_{ u_0, v_0} \Big[ \sup_{0\leq t \leq T}  \<{ u^\sse{\gamma}_t, \phi}_\cS^2 \Big] < \infty,\quad \sup_{\gamma>0} \E_{ u_0, v_0} \Big[ \sup_{0\leq t \leq T}  \<{ v^\sse{\gamma}_t, \phi}_\cS^2 \Big] < \infty, 
\end{equation}
\begin{equation}\label{uniform first moments dual process 2}
 \sup_{\gamma>0} \E_{ u_0, v_0} \Big[ \sup_{0\leq t \leq T} \big\langle L^\sse{\gamma}_t, \phi\big\rangle_\cS \Big] < \infty.
\end{equation}
Compare also the proof of Lemma \ref{lemma tightness} below for the strategy how to derive these estimates 
from Lemma~\ref{lemma boundedness quadratic variation}.
As in \cite[Prop.\ 3.3]{BHO15}, these estimates in turn imply the compact containment condition for the family of processes $\{( u^\sse{\gamma}_t, v^\sse{\gamma}_t,L^\sse{\gamma}_t)_{t\ge0}:\gamma>0$\}. 
Tightness in the Meyer-Zheng topology is then proved similarly to Proposition~\ref{prop:tightness} below, using the martingale problem formulation of $\mathrm{SBM}(\rho,\gamma)$ together with the bounds \eqref{uniform first moments dual process 1}-\eqref{uniform first moments dual process 2}. 
\end{proof}

Next, one has to check that limit points of the family $\{(u_t^\sse{\gamma},v_t^\sse{\gamma},L_t^\sse{\gamma})_{t\ge0}: \gamma>0\}$ solve the martingale problem $(\mathbf{MP}_F(\cS))_{u_0,v_0}^\rho$
and satisfy the separation-of-types property for positive times. The following corresponds to \cite[Prop.\ 4.3]{BHO15}:

\begin{prop}\label{MPinf 2}
Let 
$\rho\in[-1,0)$ and $( u_0, v_0)\in(\calB^+_\rap(\cS))^2$ (resp.\ $(\calB^+_\tem(\cS))^2$).
Suppose that $(u_t,v_t,L_t)_{t\ge0}\in D_{[0,\infty)}(\calM_{\rap}(\cS)^3)$ (resp.\ $D_{[0,\infty)}(\calM_{\tem}(\cS)^3)$) is any limit point with respect to the Meyer-Zheng topology of the family $\{(u^\sse{\gamma}_t,v^\sse{\gamma}_t,L^\sse{\gamma}_t)_{t\ge0}: \gamma>0\}$. 
Then for all test functions $\phi,\psi\in\calC_\tem^\ssup{2}(\cS)^+$ (resp.\ $\calC_\rap^\ssup{2}(\cS)^+$), the process
\begin{align}\label{MP8}
\bal 
\tilde M_t(\phi,\psi)&:= F(u_t, v_t,\phi,\psi) - F(u_0,v_0,\phi,\psi)\\
&\quad - \frac{1}{2}\int_0^t F(u_s, v_s,\phi,\psi)\,\langle\langle u_s, v_s, {\Delta}\phi, {\Delta}\psi\rangle\rangle_\rho \, ds \\
& \quad - 4(1-\rho^2)\int_{[0,t]\times\cS} F(u_s,v_s,\phi,\psi)\,\phi(x)\psi(x)\,L(ds,dx)
\eal
\end{align}
is a martingale,
and the process $(L_t)_{t\ge0}$ satisfies the requirements of Definition \ref{defn:MP'}.
In particular, $(u_t,v_t)_{t\ge0}$ solves the martingale problem $(\mathbf{MP}_F(\cS))_{ u_0, v_0}^\rho$. 
\end{prop}
As in the proof of \cite[Prop.\ 4.3]{BHO15}, this follows from the fact \eqref{MP8} 
holds for the finite rate model $\mathrm{SBM}(\rho,\gamma)$, by taking the limit $\gamma\to\infty$. 
Compare also the proof of Proposition \ref{prop:limit_points_MP} below.

The next lemma gives the crucial bound on mixed second moments of limit points, from which the separation-of-types property can be derived as in \cite{BHO15}.

\begin{lemma}[Moment bounds]
\label{lem:moments 2}
Let $\rho\in[-1,0)$ and $( u_0, v_0)\in(\calB^+_\rap(\cS))^2$ (resp.\ $(\calB^+_\tem(\cS))^2$).
Suppose that 
$(u_t,v_t)_{t\ge0}\in D_{[0,\infty)}(\calM_{\rap}(\cS)^2)$ (resp.\ $D_{[0,\infty)}(\calM_{\tem}(\cS)^2)$) is any limit point with respect to the Meyer-Zheng topology of the family 
$\{(u^\sse{\gamma}_t,v^\sse{\gamma}_t)_{t\ge0}:\gamma>0\}$.
Then we have for all $\phi,\psi\in\bigcup_{\lambda>0}\calC_{-\lambda}^+(\cS)$ (resp.\ $\bigcup_{\lambda>0}\calC_\lambda^+(\cS)$) that
\begin{align}\label{second mixed moment discrete}\bal
\E_{u_0,v_0}\left[\langle u_t,\phi\rangle_{\cS}\,\langle v_t,\psi\rangle_{\cS}\right] &\le\left\langle \phi\otimes\psi, {\tilde S}_t(u_0\otimes v_0)\right\rangle_{\cS^2}.
\eal\end{align}
\end{lemma}
The  mixed second moment estimate \eqref{second mixed moment discrete} is again a consequence of the colored particle moment duality and is proved exactly as in \cite[Lemma 4.4]{BHO15}, 
see in particular inequality (51) there.
Of course, by Proposition\ \ref{prop:moments}a) in fact equality holds in \eqref{second mixed moment discrete}, 
which however we can prove only in Subsection \ref{ssec:further_properties} below. 

\begin{corollary}[Separation of Types]\label{cor:sep discrete}
Under the assumptions of Lemma\ \ref{lem:moments 2}, we have for all $t>0$ the separation-of-types property \eqref{eq:dsep} if $\cS=\Z^d$ resp.\ \eqref{eq:csep} if $\cS=\R$.
\end{corollary}
{\begin{proof}
For the continuous case $\cS=\R$, see \cite[Lemma 4.4]{BHO15} and also the proof of Cor.\ \ref{cor:sep} below. For the discrete case $\cS=\Z^d$ and $k\in\cS$, simply choose $\phi:=\psi:=\1_{\{k\}}$ in \eqref{second mixed moment discrete}.  
\end{proof}

By the above results, it is straightforward to show uniqueness in our martingale problem (under the separation-of-types condition), which as in \cite{BHO15} follows from self-duality. 
Note that up to now, all results on tightness and properties of limit points included also the case $\rho=-1$. However, in the next proposition we have to exclude this case since for $\rho=-1$ the self-duality is no longer sufficient to deduce uniqueness.

\begin{prop}[Uniqueness]\label{prop:uniqueness}
Fix $\rho\in(-1,0)$ and (possibly random) initial conditions $(u_0,v_0)\in\calM_\tem(\cS)^2$ or $\calM_\rap(\cS)^2$. 
Then there is at most one solution $(u_t,v_t)_{t\ge0}$ to the martingale problem $(\mathbf{MP}_F(\cS))_{u_0,v_0}^\rho$ 
satisfying the separation-of-types property \eqref{eq:dsep} if $\cS=\Z^d$ resp.\ \eqref{eq:csep} if $\cS=\R$.
\end{prop}

\begin{proof}
For the case $\cS=\R$, this is proved in \cite[Prop.\ 5.2]{BHO15}, and the case $\cS=\Z^d$ follows along the same lines.
As in \cite[Prop.\ 5.1]{BHO15}, one shows first that solutions to $(\mathbf{MP}_F(\Z^d))^\rho$ satisfying the separation-of-types condition are self-dual w.r.t.\ the function $F$ from \eqref{self-duality function}: 
For any solution $(u_t,v_t)_{t\geq 0}$ of the martingale problem $(\mathbf{MP}_F(\Z^d))_{u_0,v_0}^\rho$ with initial conditions
$(u_0,v_0) \in \calM_\tem(\Z^d)^2$ and any solution $(\tilde u_t, \tilde v_t)_{t \geq 0}$ of
$(\mathbf{MP}_F(\Z^d))_{\tilde u_0,\tilde v_0}^\rho$ with $(\tilde u_0,\tilde v_0) \in \calM_\rap(\Z^d)^2$, we have
\begin{equation}\label{eq:self-duality}
 \E_{u_0, v_0} [ F(u_t,v_t,\tilde u_0, \tilde v_0)] =\E_{\tilde u_0, v_0} [ F(u_0,v_0,\tilde u_t, \tilde v_t)], \qquad t \geq 0.
 \end{equation}
In fact, for the discrete case the proof simplifies considerably: 
Since $\cM_\tem(\Z^d)=\cB_\tem^+(\Z^d)$ and the discrete Laplace operator can be applied directly to the solution $(u_t,v_t)$, 
we do not need to perform a spatial smoothing via the heat kernel $S_\eps$ as in the proof of \cite[Prop.\ 5.1]{BHO15}. 
See also the proof of \cite[Prop.\ 4.7]{KM11b} for the slightly different martingale problem employed in that paper, or the proof of \cite[Thm.\ 2.4(b)]{DP98} for the discrete finite rate model. With the self-duality at hand, uniqueness follows by standard arguments, see e.g.~\cite[proof of Prop.\ 4.1]{KM11b} or \cite[proof of Thm.\ 2.4(a)]{DP98}.
\end{proof}

We are now ready to finish the proof of Theorem~\ref{thm:MPinf_discrete}:
For $( u_0, v_0)\in(\calB^+_\tem(\cS))^2$ or $(\calB^+_\rap(\cS))^2$, combining Prop.~\ref{tightness_MZ}, Prop.\ \ref{MPinf 2}, Cor.\ \ref{cor:sep discrete} and Prop.~\ref{prop:uniqueness} yields convergence of the finite-rate
models $\mathrm{SBM}(\rho,\gamma)_{u_0,v_0}$ to $\mathrm{SBM}(\rho,\infty)_{u_0,v_0}$ as $\gamma\uparrow\infty$, 
and in particular also existence and uniqueness (subject to separation-of-types) of solutions to the martingale problem $(\mathbf{MP}_F(\cS))_{u_0,v_0}^\rho$. Thus Theorem \ref{thm:MPinf_discrete}b) is fully proved. 

For part a) and $\cS=\R$, it remains to show existence of a solution to 
$(\mathbf{MP}_F(\R))_{u_0,v_0}^\rho$ satisfying the separation-of-types condition \eqref{eq:csep} if the initial conditions are from $\cM_\tem(\R)^2$ rather than from $\cB_\tem^+(\R)^2$.
This can be done by approximating $(u_0,v_0)\in\cM_\tem(\R)^2$ by absolutely continuous initial conditions $(S_\eps u_0,S_\eps v_0)\in\cB_\tem^+(\R)^2$ 
and then showing tightness, properties of limit points and convergence as $\eps\downarrow0$ by the same strategy as above.
Alternatively, one can also do the $\gamma\uparrow\infty$- and $\eps\downarrow0$-limits `at the same time' 
and repeat all arguments in this subsection with variable initial conditions $(u_0^\sse{\gamma}, v_0^\sse{\gamma}):=(S_{\gamma^{-1}}u_0,S_{\gamma^{-1}}v_0)$, $\gamma>0$, instead of fixed ones.
Finally, the existence of a solution for initial conditions from $\cM_\tem(\R)^2$ follows also from the proof of our convergence result Theorem \ref{thm:conv}, see Section \ref{sec:convergence} below.

Finally, in order to show the strong Markov property we argue as in the proof of \cite[Lemma 5.4]{Dawsonetal2003}: 
We know that for each $(u_0,v_0)\in\cM_\tem(\cS)^2$ resp.\ $\cM_\rap(\cS)^2$, there is a unique solution $(u_t,v_t)_{t\ge0}$ to $(\mathbf{MP}_F(\cS))_{u_0,v_0}^\rho$. For $t>0$, let $P_t\big((u_0,v_0);\cdot\,\big)$ denote the law of $(u_t,v_t)$ on $\cM_\tem(\cS)^2$ resp.\ on $\cM_\rap(\cS)^2$ under the corresponding probabilty measure $\p_{u_0,v_0}$. 
Using the self-duality \eqref{eq:self-duality} for $\cS=\Z^d$ resp.\ the approximate self-duality \cite[eq. (56)]{BHO15} for $\cS=\R$ and a monotone class argument, it is easy to see 
that $(u_0,v_0)\mapsto P_t((u_0,v_0);\cdot\,)$ is Borel measurable; consequently $P_t$ is a transition kernel. Now the strong Markov property of $(u_t,v_t)_{t\ge0}$ follows along the same lines as in the proof of \cite[Lemma 5.4]{Dawsonetal2003}.

\subsection{Further properties of the limit}\label{ssec:further_properties}
In this subsection, we prove some additional properties of the infinite rate model $\mathrm{SBM}(\rho,\infty)$, including the identities for second moments from Proposition~\ref{prop:moments}.

We start by proving a version of Lemma 3.2 in \cite{DM11a} that bounds $p$-th moments 
of $\mathrm{SBM}(\rho,\gamma)$ uniformly in $\gamma$, provided $\rho$ is `sufficiently close' to $-1$. 
This result is stated in \cite{DM11a} for $p=2+\eps$ (in which case it holds for all $\rho<0$, and which is all that we will need in the present paper\footnote{In \cite{HOV15} however, we use the result also for values $n>2$.}), 
but using the `critical curve' of \cite{BDE11} the proof works in fact for other values of $p$ as well.}
Also in contrast to \cite{DM11a}, we do not restrict to integrable initial conditions and do not consider the total mass, but test against suitable test functions. 

In the following, for a $\rho$-correlated planar Brownian motion $(W^\ssup{1},W^\ssup{2})$ 
starting at $(x,y)\in(\R^+)^2$, we denote by $\tau$ the first hitting time of $(W^\ssup{1},W^\ssup{2})$ at the boundary of the first quadrant $(\R^+)^2$, and by  $\E_{x,y}[\cdot]$ the corresponding expectation.

\begin{lemma}\label{lem:DDS2}
Let $\rho\in[-1,1)$ and $p\ge 1$ such that $\rho+\cos(\pi/p)<0$. 
Then there exists a constant $C=C(p)$ only depending on $p$ such that the following holds: 
For all $(u_0,v_0)\in\calB^+_\tem(\cS)^2$ (resp.\ $\calB^+_\rap(\cS)^2$), $\phi\in\bigcup_{\lambda>0}\calC_\lambda^+(\cS)$ (resp.\ $\bigcup_{\lambda>0}\calC_{-\lambda}^+(\cS)$) and $T>0$ we have
\[\bal
&\sup_{\gamma>0}\E_{u_0,v_0}\left[\sup_{t\in[0,T]}\<{ u^\sse{\gamma}_t, {}S_{T-t}\phi}_\cS^{p}\right] \\
&\le C\left(\E_{\langle u_0, {S}_T\phi\rangle_{\cS},\langle v_0, {S}_T\phi\rangle_{\cS}}\left[\tau^{{p}/2}\right] + \<{ u_0, {S}_T\phi}_\cS^{p}\right)<\infty,
\eal\]
and analogously for $v^\sse{\gamma}$.
\end{lemma}
\begin{proof}
Fix $T>0$. 
By the Green function representation of the finite rate model (see \eqref{MP1a dual}-\eqref{Cov1a dual}), we know that  
\begin{align}\label{MP2a dual}
\bal
 M^\sse{\gamma,T}_t(\phi) &:= \<{ u^\sse{\gamma}_t, {}S_{T-t}\phi}_{\cS} - \<{ u_0, {S}_T\phi}_{\cS},\\
 N^\sse{\gamma,T}_t(\phi) &:= \<{ v^\sse{\gamma}_t, {}S_{T-t}\phi}_{\cS} - \<{ v_0, {S}_T\phi}_{\cS},
\eal\end{align}
$t\in[0,T]$, are continuous square-integrable zero-mean martingales with covariation structure given by
\[\bal
  &[ M^\sse{\gamma,T}(\phi),  M^\sse{\gamma,T}(\phi) ]_t \notag\\ &
  \hspace{1cm} = [ N^\sse{\gamma,T}(\phi), N^\sse{\gamma,T}(\phi) ]_t 
  = \gamma \int_0^t \< {\left({S}_{T-r}\phi\right)^2, u_r^\sse{\gamma}v_r^\sse{\gamma}}_\cS\,dr, 
  \\
& [  M^\sse{\gamma,T}(\phi) , N^\sse{\gamma,T}(\phi) ]_t   = \rho \gamma  \int_0^t \< {\left({S}_{T-r}\phi\right)^2, u_r^\sse{\gamma}v_r^\sse{\gamma}}_\cS\,dr.
\eal\]
By a version of the Dubins-Schwarz Theorem (see e.g.\ \cite[Lemma 4.2, Rem.\ 4.3]{BDE11}), we can represent $(M^\sse{\gamma,T}_t(\phi), N^\sse{\gamma,T}_t(\phi))_{t\in[0,T]}$ 
(possibly on an enlargement of the underlying probability space) as a time-changed $\rho$-correlated planar Brownian motion $(W^\ssup{1},W^\ssup{2})$, i.e.
\[(M^\sse{\gamma,T}_t(\phi), N^\sse{\gamma,T}_t(\phi))=\left(W^\ssup{1}_{A_t},W^\ssup{2}_{A_t}\right),\qquad t\in[0,T],\]
where $A_t:=\gamma\int_0^t\left\langle ({S}_{T-s}\phi)^2,u_s^\sse{\gamma} v_s^\sse{\gamma}\right\rangle_\cS\, ds$.
Now let $\tau_Q:=\inf\{t>0:(W_t^\ssup{1},W_t^\ssup{2})\notin Q\}$ denote  the first hitting time of $(W^\ssup{1},W^\ssup{2})$ of the boundary of the quadrant
\[Q:=\left\{(x,y)\in\R^2:x\ge -\<{ u_0, {S}_T\phi}_{\cS} \text{ and } y\ge -\<{ v_0, {S}_T\phi}_{\cS}\right\}.\]
Since $u^\sse{\gamma}$, $v^\sse{\gamma}$ and $\phi$ are all nonnegative, we have 
$(M^\sse{\gamma,T}_t(\phi), N^\sse{\gamma,T}_t(\phi))\in Q$ and thus $A_t\le\tau_Q$ for all $t\in[0,T]$. 
Then for each $p\ge1$ we obtain by \eqref{MP2a dual} and the Burkholder-Davis-Gundy inequality that
\begin{align}\label{key_estimate}\bal
\E_{u_0,v_0}\left[\sup_{t\in[0,T]}\<{ u^\sse{\gamma}_t, {}S_{T-t}\phi}_\cS^p\right] &\le C_p\left(\E_{u_0,v_0}\left[\sup_{t\in[0,T]}\left|M_t^\sse{\gamma,T}(\phi)\right|^p\right] + \<{ u_0, {S}_T\phi}_\cS^p\right) \\
&\le C_p'\left(\E_{0,0}\left[\tau_Q^{p/2}\right] + \<{ u_0, {S}_T\phi}_\cS^p\right) \\
&=C_p'\left(\E_{\langle u_0, {S}_T\phi\rangle_{\cS},\langle v_0, {S}_T\phi\rangle_{\cS}}\left[\tau^{p/2}\right] + \<{ u_0, {S}_T\phi}_\cS^p\right),
\eal\end{align}
and the last expectation is finite iff $\rho+\cos(\frac{\pi}{p})<0$, see \cite[Thm.\ 5.1]{BDE11}. 
Since the constant $C_p'$ depends only on $p$, the proof is complete.
\end{proof}
By combining the previous lemma with the lower bounds in Lemma \ref{lemma estimates}a) for continuous space (see \eqref{estimate 1a}) 
resp.\ Lemma \ref{lem:uniform preservation} for discrete space (choose $n=1$ in \eqref{eq:uniform preservation}), the following corollary is immediate, 
where we recall the notation $\phi_\la(x) := e^{-\la |x|}$ for $x \in \cS$, $\la\in\R$. 

\begin{corollary}\label{cor:uniform p-th moment 2}
Let $\rho\in[-1,1)$ and $p\ge 1$ such that $\rho+\cos(\pi/p)<0$. 
Then for each $T>0$ and $\lambda>0$ (resp.\ $\lambda<0$) there exists a constant $C(p,\lambda,T)$ such that
for all $(u_0,v_0)\in\calB^+_\tem(\cS)^2$ (resp.\ $\calB^+_\rap(\cS)^2$) we have
\[\bal
&\sup_{\gamma>0}\E_{u_0,v_0}\left[\sup_{t\in[0,T]}\<{ u^\sse{\gamma}_t, \phi_\lambda}_\cS^{p}\right]\\
&\le C(p,\lambda,T)\left(\E_{\langle u_0, {S}_T\phi_\lambda\rangle_\cS,\langle v_0, {S}_T\phi_\lambda\rangle_\cS}\left[\tau^{p/2}\right] + \<{ u_0, {S}_T\phi_\lambda}_\cS^{p}\right)<\infty,
\eal\]
and analogously for $v^\sse{\gamma}$.
\end{corollary}

Applied for $\rho<0$ and $p>2$, the above uniform moment bounds allow us to prove Proposition \ref{prop:moments}.  
They also allow us to extend the martingale representation \eqref{eq:discrete MP 1_finite}-\eqref{eq:discrete MP 1a_finite} and the Green function representation \eqref{MP1a dual}-\eqref{Cov1a dual} to the infinite rate limit. Before turning to the proof of Prop.\ \ref{prop:moments}, we collect these and some additional properties of the limit (which will be of importance in Section \ref{sec:convergence} below) in the following proposition:

\begin{prop}\label{prop:further_properties}
Let $\rho\in(-1,0)$ and $( u_0, v_0)\in\calB^+_\rap(\cS)^2$ (respectively $\calB^+_\tem(\cS)^2$).
Then as $\gamma\uparrow\infty$, the processes $(u^\sse{\gamma}_t,v^\sse{\gamma}_t,L_t^\sse{\gamma})_{t\ge0}$ converge with respect to the Meyer-Zheng topology to a process $(u_t,v_t,L_t)_{t\ge0}\in D_{[0,\infty)}(\calM_{\rap}(\cS)^3)$ (respectively $D_{[0,\infty)}(\calM_{\tem}(\cS)^3)$). 
The limit has the following properties in addition to those stated in Theorem \ref{thm:MPinf_discrete} and Proposition \ref{prop:moments}:
\begin{itemize}
\item[a)] 
For all $\phi,\psi\in\bigcup_{\lambda>0}\calC^\ssup{2}_{-\lambda}(\cS)$ (resp.\ $\bigcup_{\lambda>0}\calC^\ssup{2}_\lambda(\cS)$) we have that
\begin{align}\label{eq:discrete MP 1}\bal
M_t(\phi) &:= \langle u_t, \phi \rangle_\cS - \langle u_0, \phi \rangle_\cS - \frac{1}{2}\int_0^t\langle u_s, {\Delta}\phi \rangle_\cS\, ds , \\
N_t(\psi) &:= \langle v_t, \psi \rangle_\cS - \langle v_0, \psi \rangle_\cS - \frac{1}{2}\int_0^t\langle v_s, {\Delta}\psi \rangle_\cS\, ds 
\eal\end{align}
are square-integrable martingales with quadratic (co-)variation
\begin{align}\bal\label{eq:discrete MP 1a} 
{}  [ M(\phi),  M(\phi) ]_t  = [ N(\phi), N(\phi) ]_t &= \langle L_t,\phi^2\rangle_\cS, 
  \\
 [  M(\phi) , N(\psi) ]_t  & = \rho \, \langle L_t,\phi\psi\rangle_\cS
\eal\end{align}
for all $t>0$. 
\item[b)][Green function representation]\\
For all $T>0$ and $\phi,\psi\in\bigcup_{\lambda>0}\calC_{-\lambda}(\cS)$ (resp.\ $\bigcup_{\lambda>0}\calC_\lambda(\cS)$) we have that
\begin{equation}\label{eq:discrete GFR 1}
\langle u_t,{S}_{T-t}\phi\rangle_\cS=\langle u_0, {S}_T\phi\rangle_\cS + {M}_t^T(\phi),\qquad \langle v_t,{S}_{T-t}\psi\rangle_\cS=\langle v_0, {S}_T\psi\rangle_\cS + {N}_t^T(\psi)
\end{equation}
for $t\in[0,T]$, where $({M}_t^T(\phi))_{t\in[0,T]}$ and $({N}_t^T(\psi))_{t\in[0,T]}$ are square-integrable martingales
with quadratic (co-)variation
\begin{align}\bal\label{eq:discrete GFR 1a} 
{}  [ M^T(\phi),  M^T(\phi) ]_t  = [ N^T(\phi), N^T(\phi) ]_t &= \int_{[0,t]\times\cS}(S_{T-r}\phi(x))^2\,L(dr,dx), 
  \\
 [  M^T(\phi) , N^T(\psi) ]_t  & = \rho \, \int_{[0,t]\times\cS}S_{T-r}\phi(x)S_{T-r}\psi(x)\,L(dr,dx)
\eal\end{align}
for $t\in[0,T]$. In particular, we have the uniform second moment bound
\begin{align}\label{eq:discrete GFR 2}\bal
\E_{u_0,v_0}\bigg[\sup_{t\in[0,T]}|M_t^T(\phi)|^2\bigg]\le\frac{4}{ |\rho|}\left\langle \phi\otimes\phi, ({S}_T^\ssup{2} - {\tilde S}_T)(u_0\otimes v_0)\right\rangle_{\cS^2},
\eal\end{align}
and analogously for $N^T(\psi)$.
\item[c)] 
The quadratic variation of the (complex-valued) martingale $\big(\tilde M_t(\phi,\psi)\big)_{t\ge0}$ in \eqref{MP8} is given by
\begin{align}\label{qv2}\bal
&\left[ M_\cdot(\phi,\psi), \overline{ M_\cdot(\phi,\psi)}\right]_t\\
&\qquad=4(1-\rho^2)\, \int_{[0,t]\times\cS} |F(u_s,v_s,\phi,\psi)|^2\,\left(\phi(x)^2+\psi(x)^2\right)\,L(ds,dx)
\eal\end{align}
for all $t>0$.
\item[d)] We have the following first moment formula for the process $(L_t)_{t\ge0}$: 
For all $t>0$ and $\phi,\psi\in\bigcup_{\lambda>0}\calC^+_{-\lambda}(\cS)$ (resp.\ $\bigcup_{\lambda>0}\calC^+_\lambda(\cS)$),
\begin{align}\label{first moment Lambda discrete equality}\bal%
&\E_{u_0,v_0}\left[\int_{[0,t]\times\cS}{S}_{t-s}\phi(x)\,{S}_{t-s}\psi(x)\,L(ds,dx)\right]
=\frac{1}{ |\rho|}\left\langle \phi\otimes\psi, ({S}_t^\ssup{2} - {\tilde S}_t)(u_0\otimes v_0)\right\rangle_{\cS^2}.
\eal\end{align}
\end{itemize}
\end{prop}
\begin{proof}
We give the proof for $(u_0,v_0)\in\calB^+_\tem(\cS)^2$. By Prop.\ \ref{tightness_MZ}, we know that the family $\{(u_t^\sse{\gamma},v_t^\sse{\gamma},L_t^\sse{\gamma})_{t\ge0}: \gamma>0\}$ is tight with respect to the Meyer-Zheng topology on $D_{[0,\infty)}(\calM_\tem(\cS)^3)$.
Suppose that $(u_t,v_t,L_t)_{t\ge0}\in D_{[0,\infty)}(\calM_{\tem}(\cS)^3)$ is any limit point. 

First of all, we note from the proof of Lemma \ref{lem:DDS2} that for $\eps=\eps(\rho)>0$ sufficiently close to zero we have
\[\sup_{\gamma>0}\E_{u_0,v_0}\left[\bigg(\int_{[0,T]\times\cS}(S_{T-s}\phi(x))^2\,L^\sse{\gamma}(ds,dx)\bigg)^{(2+\eps)/2}\right]<\infty\]
for each $T>0$ and $\phi\in\bigcup_{\lambda>0}\calC_\lambda(\cS)$, 
see in particular \eqref{key_estimate} (with $p=2+\eps$ chosen such that $\rho+\cos(\tfrac{\pi}{2+\eps})<0$).
Using the lower bound \eqref{estimate 1a} in Lemma \ref{lemma estimates}a), we get also
\begin{align}\label{key_estimate_2}\bal
&\sup_{\gamma>0}\E_{u_0,v_0}\left[\left\langle L_T^\sse{\gamma},\phi^2\right\rangle_\cS^{(2+\eps)/2}\right]\\
&\quad\le C(\phi,\lambda,T)\,\sup_{\gamma>0}\E_{u_0,v_0}\left[\bigg(\int_{[0,T]\times\cS}(S_{T-s}\phi_\lambda(x))^2\,L^\sse{\gamma}(ds,dx)\bigg)^{(2+\eps)/2}\right]<\infty
\eal\end{align}
for all $T>0$, $\phi\in\bigcup_{\lambda>0}\calC_\lambda(\cS)$ and a suitable $\lambda=\lambda(\rho)>0$.

Now let $(M^\sse{\gamma}(\phi), N^\sse{\gamma}(\phi))$ be the martingales corresponding to \eqref{eq:discrete MP 1} 
for finite $\gamma>0$, as defined in \eqref{eq:discrete MP 1_finite}.
Then (along a subsequence which we do not distinguish in notation) $(M_t^\sse{\gamma}(\phi), N_t^\sse{\gamma}(\phi))_{t\ge0}$ converges to $(M_t(\phi), N_t(\phi))_{t\ge0}$ w.r.t.\ Meyer-Zheng, 
and for each $t>0$ we have by \eqref{eq:discrete MP 1a_finite} and \eqref{uniform first moments dual process 2} that
\[\sup_{\gamma>0}\E_{u_0,v_0}\left[|M^\sse{\gamma}_t(\phi)|^2\right]=\sup_{\gamma>0}\E_{u_0,v_0}\left[\langle L^\sse{\gamma}_t, \phi^2\rangle_\cS\right]<\infty.\]
Applying \cite[Thm.\ 11]{MZ84} we get that the limit $(M(\phi),N(\phi))$ is again a martingale. 
But in fact, by the Burkholder-Davis-Gundy inequality and \eqref{key_estimate_2} we know even that for $\eps=\eps(\rho)>0$ small enough we have
\[\sup_{\gamma>0}\E_{u_0,v_0}\left[|M^\sse{\gamma}_t(\phi)|^{2+\eps}\right]<\infty,\]
for all $t>0$. 
Thus $M_t^\sse{\gamma}(\phi)$ converges to $M_t(\phi)$ in $L^2$. Consequently, we get for the quadratic variation of the martingales that $[M^\sse{\gamma}(\phi), M^\sse{\gamma}(\phi)]_t=\langle L_t^\sse{\gamma},\phi^2\rangle_\cS$ converges to $[M(\phi), M(\phi)]_t$ in $L^1$, 
and we conclude that $[M(\phi), M(\phi)]_t=\langle L_t,\phi^2\rangle_\cS$. The proof for the covariation is the same. This is a).

Note that this implies in particular that the limit point $(L_t)_{t\ge0}$ of the family $\{(L_t^\sse{\gamma})_{t\ge0}:\gamma>0\}$ is unique, since it is characterized by the covariation structure \eqref{eq:discrete MP 1a} of the martingales in \eqref{eq:discrete MP 1} which in turn are uniquely determined by $(u_t,v_t)_{t\ge0}$, for which we have uniqueness by Prop.\ \ref{prop:uniqueness}. 
Thus we have now also proved that $L^\sse{\gamma}$ converges in $D_{[0,\infty)}(\cM_\tem(\R))$ w.r.t.\ the Meyer-Zheng topology, as in the statement of the proposition. 

The proof of the Green function representation in b) is very similar to a), using the corresponding property of the finite rate model and 
uniform boundedness of $(2+\eps)$-th moments of the martingales $(M^\sse{\gamma,T}(\phi),N^\sse{\gamma,T}(\phi))$ from \eqref{MP1a dual}-\eqref{Cov1a dual},
which again follows from \eqref{key_estimate_2}. 
The uniform second moment bound \eqref{eq:discrete GFR 2} is a direct consequence of \eqref{eq:discrete GFR 1}-\eqref{eq:discrete GFR 1a}  and the Burkholder-Davis-Gundy inequality.

Similarly, for c) we use that the analogue of formula \eqref{qv2} holds for the finite rate model, as a straightforward calculation using It\^{o}'s formula shows. Then again, estimate \eqref{key_estimate_2} allows us to extend the formula to the limit.

Finally, for d) we note that \eqref{key_estimate_2} allows us to pass to the limit in formula 
\eqref{Cov2 dual}, yielding the first moment expression \eqref{first moment Lambda discrete equality} for $L$. 
\end{proof}

Now the moment properties in Proposition \ref{prop:moments} can be proved easily:
\begin{proof}[Proof of Proposition~\ref{prop:moments}] 
Let $\rho\in(-1,0)$. Then we can choose $\eps=\eps(\rho)>0$ sufficiently small such that $\rho+\cos(\frac{\pi}{2+\eps})<0$. Applying  Corollary \ref{cor:uniform p-th moment 2} with $p=2+\eps$,
the families $\left\{\<{ u^\sse{\gamma}_t, \phi}_\cS^2:\gamma>0\right\}$ and $\left\{\<{ u^\sse{\gamma}_t, \phi}_\cS^2:\gamma>0\right\}$ are uniformly integrable, hence (by H\"older's inequality) also $\left\{\<{ u^\sse{\gamma}_t, \phi}_\cS\<{ v^\sse{\gamma}_t, \psi}_\cS:\gamma>0\right\}$. 
Together with convergence of the finite rate processes to the infinite rate process $(u_t,v_t)_{t\ge0}$, we get
\[\E_{u_0,v_0}\left[\<{ u_t, \phi}_\cS\<{ v_t, \psi}_\cS\right]=\lim_{\gamma\uparrow\infty}\E_{u_0,v_0}\left[\<{ u^\sse{\gamma}_t, \phi}_\cS\<{ v^\sse{\gamma}_t, \psi}_\cS\right]=\big\langle \phi\otimes\psi, {\tilde S}_t(u_0\otimes v_0)\big\rangle_{\cS^2},
\]
where the last equality follows directly by taking the limit in the finite rate moment duality from \cite{EF04} and the definition of the semigroup $({\tilde S}_t)_{t\ge0}$. 
This proves the mixed second moment formula \eqref{second mixed moment 2}. 
The second moment formulae \eqref{second moment 2} can be derived in a similar way from the finite rate moment duality. 
Alternatively, it follows also directly from Prop.\ \ref{prop:further_properties} b) and d), by putting $t=T$ in the Green function representation \eqref{eq:discrete GFR 1}-\eqref{eq:discrete GFR 1a} and using the covariance structure of the martingales together with the first moment formula \eqref{first moment Lambda discrete equality} for $L$.

Finally, the bound \eqref{finite p-moment} on $p$-th moments follows by an application of Fatou's lemma upon letting $\gamma\uparrow\infty$ in Corollary \ref{cor:uniform p-th moment 2}. 
Thus we have now fully proved Proposition \ref{prop:moments}.
\end{proof}

\section{Convergence of the discrete to the continuous model}\label{sec:convergence}

In this section, we prove Theorem \ref{thm:conv}.\ Recall that given initial conditions $(\mu_0,\nu_0)\in\calM_\tem(\R)^2$ (resp.\ $\calM_\rap(\R)^2$) for $\mathrm{cSBM}(\rho,\infty)$, we define $(\mu_t^\ssup{n},\nu_t^\ssup{n})_{t\ge0}$ by \eqref{defn:variabel initial conditions}-\eqref{defn:rescaling}, and our goal is to show that $(\mu^\ssup{n}_t,\nu^\ssup{n}_t)_{t\ge0}\xrightarrow{d}(\mu_t,\nu_t)_{t\ge0}$, $n\to\infty$, as measure-valued processes, where $(\mu_t,\nu_t)_{t\ge0}$ denotes the (unique) solution to $\mathrm{cSBM}(\rho,\infty)_{\mu_0,\nu_0}$ from Theorem \ref{thm:MPinf_discrete} with $\cS=\R$.
The general strategy is familiar and similar to Section \ref{sec:discrete}: First we prove tightness, 
then we show that limit points solve the martingale problem $(\mathbf{MP}_F(\R))_{\mu_0,\nu_0}^\rho$ from Definition \ref{defn:MP'} and the separation-of-types property \eqref{eq:csep}.

The proof will consist of a series of lemmas and propositions. We begin with some preliminaries. Note that if $\phi\in\bigcup_{\lambda>0}\calB_\lambda(\R)$ (resp.\ $\bigcup_{\lambda>0}\calB_{-\lambda}(\R)$) is a test function and for each $n\in\N$ we define a function $\phi^\ssup{n}:\Z\to\R$ by
\begin{equation}\label{eq:phi^n}
\phi^\ssup{n}(k) := \frac{1}{n}\phi(k/n),\qquad k\in\Z,
\end{equation}
then $\phi^\ssup{n}\in\bigcup_{\lambda>0}\calB_\lambda(\Z)$ (resp.\ $\bigcup_{\lambda>0}\calB_{-\lambda}(\Z))$ and by \eqref{defn:rescaling}
\begin{equation}\label{eq:rescaling}
 \left\langle \mu_t^\ssup{n},  \phi \right\rangle_\R 
= \left\langle u^\ssup{n}_{n^2t},\phi^\ssup{n}\right\rangle_\Z, \qquad t\ge0,
\end{equation}
and analogously for $\nu_t^\ssup{n}$. 
Recall that for each $n\in\N$, we suppose that $(u^\ssup{n}_t,v^\ssup{n}_t)_{t\ge0}$ is the unique solution to the martingale problem $(\mathbf{MP}_F(\Z))^\rho_{u_0^\ssup{n},v_0^\ssup{n}}$ by Theorem \ref{thm:MPinf_discrete}, with initial conditions defined by \eqref{defn:variabel initial conditions}. 
Thus there exists an increasing c\`adl\`ag $\calM_\tem(\Z)$-valued (resp.\ $\cM_\rap(\Z)$-valued)
process $(L_t^\ssup{n})_{t\ge0}$ satisfying the properties \eqref{finiteness Lambda 1} and \eqref{MP9} of Definition \ref{defn:MP'}. 
In fact, we know by Theorem \ref{thm:MPinf_discrete}b) that $(u_t^\ssup{n},v_t^\ssup{n})_{t\ge0}$ is given as the $\gamma\uparrow\infty$-limit of the finite rate symbiotic branching processes $\mathrm{dSBM}(\rho,\gamma)_{u_0^\ssup{n},v_0^\ssup{n}}$.
Moreover, by Proposition \ref{prop:further_properties} we can take $L^\ssup{n}$ as the $\gamma\uparrow\infty$-limit
of the corresponding finite rate increasing processes defined in \eqref{defn:L^gamma}, for which we have the first moment formula \eqref{first moment Lambda discrete equality}. 

For each $n\in\N$, we define an increasing c\`adl\`ag process $(\Lambda_t^\ssup{n})_{t\ge0}$ taking values in $\calM_\tem(\R)$ (resp.\ $\cM_\rap(\R)$) by
\begin{equation}\label{eq:Lambda^n}
\langle\Lambda_t^\ssup{n},\phi\rangle_\R := \frac{1}{n^2}\left\langle L^\ssup{n}_{n^2t},\phi(\tfrac{\cdot}{n})\right\rangle_\Z =\frac{1}{n}\left\langle L^\ssup{n}_{n^2t},\phi^\ssup{n}\right\rangle_\Z, \qquad t\ge0
\end{equation}
for each $\phi\in\bigcup_{\lambda>0}\calB_\lambda(\R)$ (resp.\ $\bigcup_{\lambda>0}\calB_{-\lambda}(\R)$). 
Observe that the measure $\Lambda_t^\ssup{n}$ is concentrated on the scaled lattice $\frac{1}{n}\Z$, with $\frac{1}{n^2} L_{n^2t}^\ssup{n}$ as density w.r.t.\ counting measure.

In the following, we will need to distinguish the discrete-space versions of the semigroups and generators employed in Sections \ref{intro}-\ref{sec:discrete} from their continuous-space counterparts.
Therefore, from now on we shall use the notations $({^dS}_t)_{t\ge0}$ and $^d\Delta$ for the discrete heat semigroup on the lattice $\Z$ and its generator, the discrete Laplacian. Moreover,  we will write $({^dS}_t^\ssup{2})_{t\ge0}$ for the corresponding two-dimensional heat semigroup on $\Z^2$ and $({^d\tilde S}_t)_{t\ge0}$ for the discrete version of the killed semigroup introduced in \eqref{eq: tilde S discrete}. The symbols $S_t$, $\Delta$, $S_t^\ssup{2}$ and $\tilde S_t$ will now be reserved for the continuous-space versions of the above.

\subsection{Tightness}

We start with a lemma showing that second moments of $(\mu^\ssup{n},\nu^\ssup{n})$ and first moments of $\Lambda^\ssup{n}$ are bounded uniformly in $n\in\N$.

\begin{lemma}\label{lemma tightness}
Suppose $\rho \in (-1, 0)$ and $( \mu_0,  \nu_0 ) \in \calM_\tem(\R)^2$ (resp.\ $\calM_\rap(\R)^2$). Then we have for all $T > 0$ and $\phi\in\bigcup_{\lambda>0}\calC_\lambda^+(\R)$ (resp.\ $\bigcup_{\lambda>0}\calC_{-\lambda}^+(\R)$) that
\begin{equation}\label{eq:uniform bound}
 \sup_{n\in\N}\E_{\mu_0^\ssup{n},\nu^\ssup{n}_0}\Big[\sup_{t\in[0,T]}\left\langle\phi, {\mu}^\ssup{n}_t\right\rangle_\R^{2}\Big]<\infty,\qquad  \sup_{n\in\N}\E_{\mu_0^\ssup{n},\nu^\ssup{n}_0}\Big[\sup_{t\in[0,T]}\left\langle\phi, {\nu}^\ssup{n}_t\right\rangle_\R^2\Big]<\infty,
 \end{equation}
and
\begin{equation}\label{eq:uniform bound 2}
 \sup_{n\in\N}\E_{\mu_0^\ssup{n},\nu^\ssup{n}_0}\Big[\sup_{t\in[0,T]}\left\langle\phi, {\Lambda}^\ssup{n}_t\right\rangle_\R\Big]<\infty.
 \end{equation}
\end{lemma}
\begin{proof}
We give the proof for $( \mu_0,  \nu_0 ) \in \calM_\tem(\R)^2$. Fix $T>0$ and assume w.l.o.g.\ that $\phi=\phi_\lambda$ for some $\lambda>0$. 
By Lemma~\ref{lem:uniform preservation} there is a constant $C=C(\lambda,T)$ independent of $n$ such that 
\begin{equation}\label{lemma tightness proof 00}
\phi_\lambda(\tfrac{k}{n})\le C\,{^dS}_{n^2T-s}(\phi_\lambda(\tfrac{\cdot}{n}))(k)
\end{equation}
for all $s\in[0,n^2T]$, $k\in\Z$ and $n\in\N$.
Using this together with the Green function representation for the discrete model (see Proposition\ \ref{prop:further_properties}b), with $[0,n^2T]$ in place of $[0,T]$),
we get
\begin{align}\label{lemma tightness proof 0}\bal
\left\langle {\mu}^\ssup{n}_t, \phi_\lambda\right\rangle_\R = \left\langle u^\ssup{n}_{n^2t},\phi_\lambda^\ssup{n}\right\rangle_\Z  
&\le C\,\left\langle u^\ssup{n}_{n^2t},{^dS}_{n^2(T-t)}\phi_\lambda^\ssup{n}\right\rangle_\Z \\
&=C\left(\left\langle u_0^\ssup{n}, {^dS}_{n^2T}\phi_\lambda^\ssup{n}\right\rangle_\Z + {M}_{n^2t}^{n^2T}(\phi_\lambda^\ssup{n})\right)
\eal\end{align}
for all $t\in[0,T]$,
where the second moment of the martingale term is bounded by 
\begin{align}\bal\label{lemma tightness proof 1}
\E_{u_0^\ssup{n},v^\ssup{n}_0}\Big[\sup_{t\in[0,T]}&\big|M_{n^2t}^{n^2T}(\phi_\lambda^\ssup{n})\big|^2\Big]
\\
&\le \frac{4}{ |\rho|}\left\langle \phi_\lambda^\ssup{n}\otimes\phi_\lambda^\ssup{n}, \big({^dS}_{n^2T}^\ssup{2} - {^d\tilde S}_{n^2T}\big)(u_0^\ssup{n}\otimes v_0^\ssup{n})\right\rangle_{\Z^2} \\
&\le\frac{4}{ |\rho|}\, \left\langle \phi_\lambda^\ssup{n}, {^d}S_{n^2T}u_0^\ssup{n} \right\rangle_\Z\,\left\langle \phi_\lambda^\ssup{n}, {^dS}_{n^2T}v_0^\ssup{n} \right\rangle_\Z
\eal\end{align}
for all $n\in\N$ (recall estimate \eqref{eq:discrete GFR 2}). 
Combining \eqref{lemma tightness proof 0}-\eqref{lemma tightness proof 1} with Lemma \ref{lem:semigroup convergence} in the appendix, 
the first inequality in \eqref{eq:uniform bound} follows easily, and the proof of the second one is analogous. 

For the increasing process $\Lambda^\ssup{n}$, we observe that 
\[\bal
 \E_{\mu_0^\ssup{n},\nu^\ssup{n}_0}\Big[\sup_{t\in[0,T]}&\left\langle\phi_\lambda, {\Lambda}^\ssup{n}_t\right\rangle_\R\Big] 
 =\frac{1}{n^2}\,\E_{u_0^\ssup{n},v^\ssup{n}_0}\left[\left\langle\phi_{\lambda}(\tfrac{\cdot}{n}), L^\ssup{n}_{n^2T}\right\rangle_\Z\right]\\
 &=\frac{1}{n^2}\,\E_{u_0^\ssup{n},v^\ssup{n}_0}\left[\int_{[0,n^2T]\times\Z}\phi_{\lambda/2}(\tfrac{\cdot}{n})(k)^2\, L^\ssup{n}(ds,dk)\right]\\
 &\le \frac{C}{n^2}\,\E_{u_0^\ssup{n},v^\ssup{n}_0}\left[\int_{[0,n^2T]\times\Z}\left({^dS}_{n^2T-s}\phi_{\lambda/2}(\tfrac{\cdot}{n})(k)\right)^2\, L^\ssup{n}(ds,dk)\right]\\
 &= \frac{C}{|\rho|}\left\langle\phi^\ssup{n}_{\lambda/2}\otimes\phi^\ssup{n}_{\lambda/2} , \big({^dS}^\ssup{2}_{n^2T}-{^d\tilde S}_{n^2T}\big)(u_0^\ssup{n}\otimes v_0^\ssup{n}) \right\rangle_{\Z^2},
\eal\]
where we used again estimate \eqref{lemma tightness proof 00} (with $\lambda/2$) for the inequality 
and formula \eqref{first moment Lambda discrete equality} for the last equality. 
Now we see as before in~\eqref{lemma tightness proof 1} that the RHS of the previous display is bounded uniformly in $n\in\N$.
\end{proof}

\begin{corollary}[Compact Containment]\label{cor:comp_cont}
Suppose $\rho \in (-1,0)$ and $ (\mu_0,  \nu_0 ) \in \calM_\tem(\R)^2$ (resp.\ $ \calM_\rap(\R)^2$).
Then the \emph{compact containment condition} holds for the family of processes $\{( \mu^\ssup{n}_t, \nu^\ssup{n}_t,\Lambda^\ssup{n}_t)_{t\ge0}:n\in\N\}$, 
i.e. for every $\eps>0$ and $T>0$ there exists a compact subset $K=K_{\eps,T}\subseteq\calM_\tem(\R)$ (resp.\ $ \calM_\rap(\R)^2$) such that 
\[\inf_{n\in\N}\p \big\{ \mu^\ssup{n}_t \in K_{\eps,T} \mbox{ for all } t \in [0,T] \big\} \ge 1-\eps,\]
and similarly for $ \nu_t^\ssup{n}$ and $ \Lambda_t^\ssup{n}$.
\end{corollary}

\begin{proof}
Given the uniform moment bounds from Lemma \ref{lemma tightness},
the proof is virtually identical to that of \cite[Corollary 3.3]{BHO15}.
\end{proof}

\begin{prop}\label{prop:tightness}
Suppose $\rho \in (-1,0)$ and $( \mu_0,  \nu_0 ) \in \calM_\tem(\R)^2$ (resp.\ $\calM_\rap(\R)^2$). 
Then the family of processes $\{(\mu_t^\ssup{n},\nu_t^\ssup{n},\Lambda_t^\ssup{n})_{t\ge0}:n\in\N\}$ 
is tight with respect to the Meyer-Zheng topology on $D_{[0,\infty)}(\calM_\tem(\R)^3)$ (resp.\ $D_{[0,\infty)}(\calM_\rap(\R)^3)$). 
\end{prop}
\begin{proof}
Suppose $( \mu_0,  \nu_0 ) \in \calM_\tem(\R)^2$. 
We will apply \cite[Cor.\ 1.4]{Kurtz91} that we recall in Appendix \ref{sec:appendix_path}, see after \eqref{MZ condition} below. 
This criterion  requires us to check the Meyer-Zheng tightness condition \eqref{MZ condition} for the coordinate processes plus a compact containment condition. 
Corollary~\ref{cor:comp_cont} takes care of the latter condition so that 
we only have to check tightness of the coordinate processes. 

Let $\phi \in \calC_\rap^\ssup{2}(\R)^+$ and fix $T>0$. 
By Prop.\ \ref{prop:further_properties}a), we know that 
\begin{align}\label{proof MZ 1}\bal
\left\langle \mu^\ssup{n}_t, \phi\right\rangle_\R &= \left\langle u^\ssup{n}_{n^2t}, \phi^\ssup{n} \right\rangle_\Z \\
&=\left\langle u_0^\ssup{n}, \phi^\ssup{n} \right\rangle_\Z + \frac{1}{2}\int_0^{n^2t}\left\langle u^\ssup{n}_{s}, {^d\Delta}(\phi^\ssup{n}) \right\rangle_\Z\, ds\ + M_{n^2t}\left(\phi^\ssup{n}\right),
\eal\end{align}
where $\left(M_{n^2t}(\phi^\ssup{n})\right)_{t\ge0}$ 
is a martingale with second moments bounded uniformly in $t\in[0,T]$ by
\[\E_{u_0^\ssup{n},v_0^\ssup{n}}\left[\sup_{t\in[0,T]}\left|M_{n^2t}\left(\phi^\ssup{n}\right)\right|^2\right]
\le4\,\E_{u_0^\ssup{n},v_0^\ssup{n}}\left[\left\langle L_{n^2T}^\ssup{n},(\phi^\ssup{n})^2\right\rangle_\Z\right],\]
where we used the Burkholder-Davis-Gundy inequality.
Choosing a suitable $\lambda>0$ and using the lower bound from Lemma \ref{lem:uniform preservation}, 
we see that there is a constant $C=C(\phi,\lambda,T)$ such that the previous display is bounded  by
\[\bal
&C\,\E_{u_0^\ssup{n},v_0^\ssup{n}}\left[ \int_{[0,n^2T]\times\Z} \left({^dS}_{n^2T-s}\phi_\lambda^\ssup{n}(k)\right)^2\,L^\ssup{n}(ds,dk)\right]\\
&=\frac{C}{|\rho|}\left\langle \phi_\lambda^\ssup{n}\otimes\phi_\lambda^\ssup{n}, \big({^dS}_{n^2T}^\ssup{2} - {^d\tilde S}_{n^2T}\big)(u_0^\ssup{n}\otimes v_0^\ssup{n})\right\rangle_{\Z^2}\\
&\le\frac{C}{ |\rho|}\,\left\langle \phi_\lambda^\ssup{n}, {^d}S_{n^2T}u_0^\ssup{n} \right\rangle_\Z\,\left\langle \phi_\lambda^\ssup{n}, {^dS}_{n^2T}v_0^\ssup{n} \right\rangle_\Z,
\eal\]
where we have also used \eqref{first moment Lambda discrete equality}. 
But the last display is bounded uniformly in $n\in\N$ by Lemma \ref{lem:semigroup convergence}, hence we get
\[\sup_{n\in\N}\sup_{t\in[0,T]}\E_{u_0^\ssup{n},v_0^\ssup{n}}\left[\left|M_{n^2t}\left(\phi^\ssup{n}\right)\right|^2\right]
<\infty\]
for all $T>0$. This implies immediately the Meyer-Zheng tightness condition \eqref{MZ condition} for the sequence of martingales in \eqref{proof MZ 1}.

In view of \eqref{proof MZ 1}, it remains to show tightness of the term 
\[X_t^\ssup{n}:=\int_0^{n^2t}\left\langle u^\ssup{n}_{s}, {^d\Delta}(\phi^\ssup{n}) \right\rangle_\Z\, ds=n^2\int_0^t\left\langle u^\ssup{n}_{n^2s}, {^d\Delta}(\phi^\ssup{n}) \right\rangle_\Z\, ds, \qquad t \in [0,T].\]
But Lemma \ref{lemma tightness} implies that this term is tight in the stronger Skorokhod topology, as follows: 
Since $\phi \in \calC_\rap^\ssup{2}(\R)^+$, there is a suitable $\lambda>0$ and some constant $C=C(\phi)$ such that 
$\left|{^d\Delta}(\phi^\ssup{n})(k)\right|\le\tfrac{C}{n^3}\phi_\lambda(\frac{k}{n})=\frac{C}{n^2}\phi_\lambda^\ssup{n}(k)$ for all $k\in\Z$, $n\in\N$. 
Thus we get for $0\le s<t\le T$ that
\begin{align}
\bal
\E_{u_0^\ssup{n},v_0^\ssup{n}}\left[|X_t^\ssup{n}-X_s^\ssup{n}|^2\right]
&\le C(\phi)\,\E_{u_0^\ssup{n},v_0^\ssup{n}}\left[\left(\int_s^t\left\langle u_{n^2r}^\ssup{n},\phi_\lambda^\ssup{n}\right\rangle_\Z\,dr\right)^2\right]\\
&\le C\,(t-s)\,\int_s^t\E_{u_0^\ssup{n},v_0^\ssup{n}}\left[\left\langle u_{n^2r}^\ssup{n},\phi_\lambda^\ssup{n}\right\rangle_\Z^2\right]dr\\
&= C\,(t-s)\,\int_s^t\E_{\mu_0^\ssup{n},\nu_0^\ssup{n}}\left[\left\langle \mu_{r}^\ssup{n},\phi_\lambda\right\rangle_\R^2\right]dr,
\eal\end{align}
where we have used Jensen's inequality in the second step. By Lemma \ref{lemma tightness}, the integrand in the above display is bounded 
uniformly in $n\in\N$ and $r \in [0,T]$, whence we get
\[\E_{u_0^\ssup{n},v_0^\ssup{n}}\left[|X_t^\ssup{n}-X_s^\ssup{n}|^2\right]\le C'\,(t-s)^2,\qquad 0\le s<t\le T,\]
confirming Kolmogorov's tightness criterion for the Laplace term.

This shows that the sequence of coordinate processes $\left(\left\langle\mu^\ssup{n}_t, \phi\right\rangle_\R\right)_{t\ge0}$, $n\in\N$, is tight w.r.t.\ the Meyer-Zheng topology. 
The same argument works for $\left(\left\langle \nu^\ssup{n}_t, \phi\right\rangle_\R\right)_{t\ge0}$. 
For the increasing process $t\mapsto\left\langle \Lambda^\ssup{n}_t, \phi \right\rangle_\R$, condition \eqref{MZ condition} reduces to 
\[\sup_{n\in\N}\E_{ \mu_0^\ssup{n}, \nu_0^\ssup{n}}\left[\left\langle \Lambda^\ssup{n}_T, \phi\right\rangle_\R\right]<\infty\] 
which is also ensured by Lemma \ref{lemma tightness}.

Since $\calC_\rap^\ssup{2}(\R)^+$ separates the points of $\cM_\tem(\R)$ and the compact containment condition holds by Corollary \ref{cor:comp_cont}, an application of \cite[Cor.\ 1.4]{Kurtz91} concludes the proof.
\end{proof}

\subsection{Properties of Limit Points }

In this subsection, we check that limit points $(\mu_t,\nu_t,\Lambda_t)_{t\ge0}$ of the sequence $(\mu^\ssup{n}_t,\nu^\ssup{n}_t,\Lambda^\ssup{n}_t)_{t\ge0}$ satisfy the martingale problem $(\mathbf{MP}_F(\R))_{ \mu_0, \nu_0}^\rho$ and the separation-of-types property \eqref{eq:csep}. 
By uniqueness in Theorem \ref{thm:MPinf_discrete}, this implies that the rescaled discrete processes converge indeed to the unique solution of $\mathrm{cSBM}(\rho, \infty)_{\mu_0,\nu_0}$. 

\begin{prop}\label{prop:limit_points_MP}
Let $\rho\in(-1,0)$ and $(\mu_0, \nu_0)\in\calM_\tem(\R)^2$ (resp.\ $\calM_\rap(\R)^2$).
Suppose $(\mu_t,\nu_t,\Lambda_t)_{t\ge0}\in D_{[0,\infty)}(\calM_{\tem}(\R)^3)$ (resp.\ $D_{[0,\infty)}(\calM_{\rap}(\R)^3)$)
is any limit point with respect to the Meyer-Zheng topology of the sequence $(\mu^\ssup{n}_t,\nu^\ssup{n}_t,\Lambda^\ssup{n}_t)_{t\ge0}$, $n\in\N$. 
Then $(\mu_t,\nu_t)_{t\ge0}$ solves the martingale problem $(\mathbf{MP}_F(\R))_{ \mu_0, \nu_0}^\rho$, 
with the process $(\Lambda_t)_{t\ge0}$ satisfying the requirements of Definition \ref{defn:MP'}.
\end{prop}
\begin{proof}
First of all, the limit point $(\Lambda_t)_{t\ge0}$ of the sequence $(\Lambda_t^\ssup{n})_{t\ge0}$ has the properties required in Definition~\ref{defn:MP'}: 
It is clear that $(\Lambda_t)_{t\ge0}$ is increasing with $\Lambda_0=0$,
and condition \eqref{finiteness Lambda 1} follows from the first moment estimate \eqref{eq:uniform bound 2} together with an application of Fatou's lemma.
It remains to check that for all 
test functions $\phi,\psi\in\calC_{\rap}^\ssup{2}(\R)^+$,
the process
\begin{align}\label{MP11}
\bal 
 M_t(\phi,\psi)  & :=   F(\mu_t, \nu_t,\phi,\psi) - F(\mu_0,\nu_0,\phi,\psi)\\
& \quad - \frac{1}{2}\int_0^t F(\mu_s, \nu_s,\phi,\psi)\,\langle\langle \mu_s, \nu_s, \Delta\phi, \Delta\psi\rangle\rangle_\rho \, ds \\
& \quad - 4(1-\rho^2)\int_{[0,t]\times\R} F(\mu_s,\nu_s,\phi,\psi)\,\phi(x)\psi(x)\,\Lambda(ds,dx),\qquad t\ge0
\eal
\end{align}
is a martingale. 

Since $(u^\ssup{n}, v^\ssup{n})$ solves the discrete martingale problem $(\mathbf{MP}_F(\Z))_{u_0^\ssup{n},v_0^\ssup{n}}^\rho$,
we know that
\[\bal  
\tilde M_{n^2t}(\phi^\ssup{n},\psi^\ssup{n})&:= F(u_{n^2t}^\ssup{n}, v_{n^2t}^\ssup{n},\phi^\ssup{n},\psi^\ssup{n}) - F(u_0^\ssup{n},v_0^\ssup{n},\phi^\ssup{n},\psi^\ssup{n})\\
& \quad - \frac{1}{2}\int_0^{n^2t} F(u_{s}^\ssup{n}, v_{s}^\ssup{n},\phi^\ssup{n},\psi^\ssup{n})\,\langle\langle u_{s}^\ssup{n}, v_{s}^\ssup{n}, {^d\Delta}(\phi^\ssup{n}), {^d\Delta}(\psi^\ssup{n})\rangle\rangle_\rho \,ds \\
& \quad - 4(1-\rho^2)\int_{[0,n^2t]\times\Z} F(u_{s}^\ssup{n},v_{s}^\ssup{n},\phi^\ssup{n},\psi^\ssup{n})\,\phi^\ssup{n}(k)\psi^\ssup{n}(k)\,L^\ssup{n}(ds,dk)\\
\eal\]
is a martingale for each $n\in\N$. 
Choose a sequence $n_k\uparrow\infty$ such that $( \mu_t^\ssup{n_k}, \nu_t^\ssup{n_k}, \Lambda_t^\ssup{n_k})_{t\ge0}$ converges to $( \mu_t, \nu_t, \Lambda_t)_{t\ge0}$ w.r.t.\ the Meyer-Zheng topology on $D_{[0,\infty)}(\calM_\tem(\R)^3)$.
In view of \eqref{eq:rescaling} and \eqref{eq:Lambda^n} (and also using the usual approximation of the continuum Laplace operator by its rescaled discrete counterpart), 
we get that $(\tilde M_{n^2t}(\phi^\ssup{n},\psi^\ssup{n}))_{t\ge0}$ converges to $( M_t(\phi,\psi))_{t\ge 0}$ w.r.t.\ Meyer-Zheng on $D_{[0,\infty)}(\R)$ as $k\to\infty$.
Now fixing $T>0$, we know by Burkholder-Davis-Gundy and \eqref{qv2} in Prop.\ \ref{prop:further_properties}c) that
\[\bal
\E_{ u_0^\ssup{n}, v_0^\ssup{n}}\left[\sup_{t\in[0,T]}|\tilde M_{n^2t}(\phi^\ssup{n},\psi^\ssup{n})|^2\right] 
&\le C\,\E_{u_0^\ssup{n},v_0^\ssup{n}}\left[ \left\langle L^\ssup{n}_{n^2T}, (\phi^\ssup{n})^2+(\psi^\ssup{n})^2\right\rangle_\Z\right],
\eal\]
where we have also used that $|F(\cdot)|\le1$. 
Now we argue as in the proof of Proposition \ref{prop:tightness}: Choosing a suitable $\lambda>0$ and combining the lower bound from Lemma \ref{lem:uniform preservation} with formula \eqref{first moment Lambda discrete equality}, 
we see that there is a constant $C'=C'(\phi,\psi,\lambda,T)$ such that the previous display is bounded  by
\[\bal
&C'\,\E_{u_0^\ssup{n},v_0^\ssup{n}}\left[ \int_{[0,n^2T]\times\Z} \left({^dS}_{n^2T-s}\phi_\lambda^\ssup{n}(k)\right)^2\,L^\ssup{n}(ds,dk)\right]\\
&=\frac{C'}{|\rho|}\left\langle \phi_\lambda^\ssup{n}\otimes\phi_\lambda^\ssup{n}, ({^dS}_{n^2T}^\ssup{2} - {^d\tilde S}_{n^2T})(u_0^\ssup{n}\otimes v_0^\ssup{n})\right\rangle_{\Z^2}\\
&\le\frac{C'}{|\rho|} \,\left\langle \phi_\lambda^\ssup{n}, {^dS}_{n^2T} u_0^\ssup{n}\right\rangle_\Z\left\langle \phi_\lambda^\ssup{n}, {^dS}_{n^2T} v_0^\ssup{n})\right\rangle_{\Z}.
\eal\]
Again by Lemma \ref{lem:semigroup convergence}, 
the last display is bounded uniformly in $n\in\N$, hence we get
\[\sup_{n\in\N}\sup_{t\in[0,T]}\E_{ u_0^\ssup{n}, v_0^\ssup{n}}\left[|\tilde M_{n^2t}(\phi^\ssup{n},\psi^\ssup{n})|^2\right]<\infty\]
for all $T>0$.
Applying \cite[Thm.\ 11]{MZ84}, we infer that the Meyer-Zheng limit $(M_t(\phi,\psi))_{t\ge0}$ is again a martingale, which completes our argument. 
\end{proof}

We now turn to the separation-of-types property. As in \cite[Lemma 4.4]{BHO15} (see in particular inequality (51) there), it can be derived from the following bound on mixed second moments:

\begin{lemma}[Mixed second moment bound]\label{lem:moments 1}
Let $\rho\in(-1,0)$ and $( \mu_0, \nu_0)\in\calM_\tem(\R)^2$ (resp.\ $\calM_\rap(\R)^2$).
Suppose $(\mu_t,\nu_t)_{t\ge0}\in D_{[0,\infty)}(\calM_{\tem}(\R)^2)$ (resp.\ $D_{[0,\infty)}(\calM_{\rap}(\R)^2)$) 
is any limit point with respect to the Meyer-Zheng topology of the sequence $(\mu^\ssup{n}_t,\nu^\ssup{n}_t)_{t\ge0}$, $n\in\N$, from \eqref{defn:rescaling}. Then with $\tilde S_t$ as defined in \eqref{eq: tilde S discrete}, we have 
\begin{align}\label{second mixed moment 1}\bal
\E_{\mu_0,\nu_0}\left[\left\langle \mu_t,\phi\right\rangle_\R\left\langle \nu_t,\psi\right\rangle_\R\right] 
&\le\big\langle \phi\otimes\psi, {\tilde S}_t( \mu_0\otimes \nu_0)\big\rangle_{\R^2}\\
\eal\end{align}
for all $t\ge0$ and $\phi,\psi\in\bigcup_{\lambda>0}\calC_{\lambda}^+(\R)$ (resp.\ $\bigcup_{\lambda>0}\calC_{-\lambda}^+(\R)$).
\end{lemma}
\begin{proof}
By \cite[Thm.\ 5]{MZ84} (see also \cite[Thm.\ 1.1(b)]{Kurtz91}) we can find a sequence $n_k\uparrow\infty$ and a set $I\subseteq(0,\infty)$ of full Lebesgue measure such that 
the finite dimensional distributions of $(\mu_t^\ssup{n_k},\nu_t^\ssup{n_k})_{t\in I}$ converge weakly to those of $(\mu_t,\nu_t)_{t\in I}$ as $k\to\infty$. 
Fix $t\in I$. Then for all test functions $\phi,\psi$ as above, we
can assume that almost surely
\begin{equation}\label{weak convergence second mixed moments}
\big\langle \mu_t^\ssup{n_k},\phi\big\rangle_\R\big\langle \nu_t^\ssup{n_k},\psi\big\rangle_\R\xrightarrow{k\uparrow\infty}\left\langle \mu_t,\phi\right\rangle_\R\left\langle \nu_t,\psi\right\rangle_\R
\end{equation}
in $\R$. 
Using Fatou's lemma, we get
\[\bal
\E_{\mu_0,\nu_0}\left[\left\langle \mu_t,\phi\right\rangle_\R\left\langle \nu_t,\psi\right\rangle_\R\right]&\le\liminf_{k\to\infty}\E_{\mu_0^\ssup{n_k},\nu_0^\ssup{n_k}}\left[\big\langle \mu_t^\ssup{n_k},\phi\big\rangle_\R\big\langle \nu_t^\ssup{n_k},\psi\big\rangle_\R\right].
\eal\]
But for all $n\in\N$ we have by 
the mixed second moment formula \eqref{second mixed moment 2} (for $\cS=\Z$) that
\begin{align}\bal
\E_{\mu_0^\ssup{n},\nu_0^\ssup{n}}\left[\left\langle \mu_t^\ssup{n},\phi\right\rangle_\R\left\langle \nu_t^\ssup{n},\psi\right\rangle_\R\right]&=\E_{u_0^\ssup{n},v_0^\ssup{n}}\left[\left\langle u^\ssup{n}_{n^2t},\phi^\ssup{n}\right\rangle_\Z\,\left\langle v^\ssup{n}_{n^2t},\psi^\ssup{n}\right\rangle_\Z\right]\\
&=\big\langle \phi^\ssup{n}\otimes\psi^\ssup{n}, {^d\tilde S}_{n^2t}(u_0^\ssup{n}\otimes v_0^\ssup{n})\big\rangle_{\Z^2}.
\eal\end{align}
As the usual discrete heat semigroup converges to its continuous counterpart under diffusive rescaling, the same holds for the killed semigroup $({^d\tilde S_t})_{t\ge0}$,
see e.g. Lemma \ref{lem:semigroup convergence} for details. 
Thus the RHS of the above display converges to the corresponding continuous quantity, namely to the RHS of \eqref{second mixed moment 1}.
This shows the estimate~\eqref{second mixed moment 1}
for all $t\in I$. Using the fact that $I$ has full Lebesgue measure together with right-continuity of the paths of $(\mu_t,\nu_t)_{t\ge0}$ and Fatou's lemma, we get the same estimate for all $t>0$.
\end{proof}

\begin{corollary}[Separation of types]\label{cor:sep}
Under the assumptions of Lemma \ref{lem:moments 1}, the separation-of-types property \eqref{eq:csep} holds for each $t>0$.
\end{corollary}
\begin{proof}
Having shown the upper bound \eqref{second mixed moment 1} for the mixed second moment, the proof of the separation-of-types property is basically the same as that of \cite[Lemma 4.4]{BHO15}: 
For each $t>0$, $x\in\R$ and $\eps>0$ fixed, letting $\phi(\cdot):=\psi(\cdot):=p_\eps(x-\cdot)$ in \eqref{second mixed moment 1} gives
\begin{align}\bal\label{proof separation 1a}
&\E_{ \mu_0, \nu_0}\left[S_\eps\mu_t(x)\,S_\eps\nu_t(x)\right]\\
&\le\iint dydz\,p_\eps(x-y)p_\eps(x-z)\,\tilde S_t( \mu_0\otimes \nu_0)(y,z)\\
&\le S_{t+\eps} \mu_0(x)\,S_{t+\eps} \nu_0(x).
\eal\end{align} 
Since $(y,z)\mapsto\tilde S_t(\mu_0\otimes\nu_0)(y,z)$ is continuous and vanishes on the diagonal, by taking $\eps\downarrow0$ in the first inequality in \eqref{proof separation 1a} 
we get $\E_{ \mu_0, \nu_0}\left[S_\eps\mu_t(x)\,S_\eps\nu_t(x)\right]\to0$, which proves our claim.
\end{proof}

Note that by combining Prop.\ \ref{prop:tightness} with Prop.\ \ref{prop:limit_points_MP} and Cor.\ \ref{cor:sep}, we have now fully proved the convergence result in Theorem \ref{thm:conv}.

\section{Preservation of order and the single-point interface}\label{sec:proof SPI}

In this section, we prove Theorem \ref{thm:SST}. 
For $\rho\in(-1,0)$, consider the solution $(\mu_t,\nu_t)_{t\ge0}$ to $\mathrm{cSBM}(\rho,\infty)_{\mu_0,\nu_0}$ with initial conditions $(\mu_0,\nu_0)\in\calM_\tem(\R)^2$ or $\cM_\rap(\R)^2$ which are mutually singular 
and such that the types are strictly ordered, i.e. $R(\mu_0)\le L(\nu_0)$. 
The proof involves several parts:
First we show that the initial ordering of types is preserved:
\begin{equation}\label{proof:SST1}
\p_{\mu_0,\nu_0}[R(\mu_t)\le L(\nu_t)\text{ for all }t>0]=1.
\end{equation}
Then, we will show that the solution has a single-point interface, i.e.
\begin{equation}\label{proof:SST2}
\p_{\mu_0,\nu_0}[R(\mu_t)= L(\nu_t)]=1
\end{equation}
for all $t>0$. 
Both properties (in an analogous sense) are first shown for the discrete model and then extended to the continuous model by an application of Theorem \ref{thm:conv}. 
The central observation in Lemma~\ref{lem:lower-/upper semicontinuity} is here that the mapping $\mu \mapsto R(\mu)$, respectively $\mu \mapsto L(\mu)$,
which assigns to each measure the rightmost, respectively leftmost, point in the support is
lower (respectively upper) semicontinuous. 
Thus, \eqref{proof:SST1} follows by arguing that the discrete approximation satisfies the same property
combined with Theorem~\ref{thm:conv}.
In order to prove \eqref{proof:SST2}, we combine \eqref{proof:SST1} with the observation that for fixed $t>0$, the measure $\mu_t+\nu_t$ is strictly positive almost surely, i.e. $\supp(\mu_t+\nu_t)=\R$, see Cor.\ \ref{cor:strict positivity} below. 
Finally, we derive the mutual singularity of the measures $\mu_t$ and $\nu_t$ for each $t>0$ from the single-point interface and the separation-of-types property \eqref{eq:csep}.

\begin{lemma}\label{lem:lower-/upper semicontinuity}
The mapping 
\[R(\cdot):\calM_\tem(\R)\to\bar\R,\qquad \mu\mapsto R(\mu):=\sup\supp(\mu)\]
is lower semicontinuous, and the mapping
\[L(\cdot):\calM_\tem(\R)\to\bar\R,\qquad \mu\mapsto L(\mu):=\inf\supp(\mu)\]
is upper semicontinuous. The same holds if $\cM_\tem(\R)$ is replaced by $\cM_\rap(\R)$.
\end{lemma}
\begin{proof}
We will prove lower semicontinuity of $R(\cdot)$, since the proof for upper semicontinuity of $L(\cdot)$ is completely analogous.

Let $(\mu^\ssup{n})_{n\in\N}$ be a sequence in $\calM_\tem(\R)$ with $\mu^\ssup{n}\to\mu\in\calM_\tem$ as $n\to\infty$.
For the purposes of the proof, we write 
\[R_n:=R(\mu^\ssup{n})=\sup\supp(\mu^\ssup{n}),\qquad \tilde R:=\liminf_{n\to\infty}R_n\in\bar\R,
\]
and we have to show that
\[R(\mu)\le \tilde R.\]
We distinguish the three possible cases $\tilde R=\infty$, $\tilde R=-\infty$ and $\tilde R\in\R$, where in the first case the assertion is trivially true. 
Suppose that $\tilde R=-\infty$, then we can find a subsequence such that $R_{n_k}\to-\infty$. Take a test function $\phi\in\calC_c^+$.
Then $\langle\mu^\ssup{n_k},\phi\rangle\to0$ since $\mu^\ssup{n_k}$ is supported on $(-\infty, R_{n_k}]$ and $\phi$ is compactly supported. Since on the other hand $\langle\mu^\ssup{n_k},\phi\rangle\to\langle\mu,\phi\rangle$, we conclude that $\langle\mu,\phi\rangle=0$. Since $\phi\in\calC_c^+$ was arbitrary, this implies 
$\mu=0$, i.e. $R(\mu)=-\infty=\tilde R$. 

It remains to consider the case $\tilde R\in\R$. We show that 
$\supp(\mu)\subseteq (-\infty, \tilde R]$,
whence we get 
$R(\mu)\le \tilde R$
and our assertion is proved. 
Assume that $\supp(\mu)\cap (\tilde R,\infty)\ne\emptyset$. Then by the definition of the support, 
we can find a function $\phi\in\calC_c^+$ with $\supp(\phi)\subseteq(\tilde R,\infty)$ such that $\langle\mu,\phi\rangle >0$. On the other hand, let $(R_{n_k})_{k\in\N}$ be a subsequence such that $R_{n_k}\to \tilde R$ as $k\to\infty$. Then we have $\langle\mu,\phi\rangle =\lim_{k\to\infty} \langle\mu^\ssup{n_k},\phi\rangle\to0 $ since $\mu^\ssup{n_k}$ is supported on $(-\infty, R_{n_k}]$ and $(-\infty, R_{n_k}]\cap\supp(\phi)=\emptyset$ for sufficiently large $k$, giving a contradiction. This completes our proof.
\end{proof}

\begin{corollary}\label{cor:preservation of SST 1}
Let $(\mu^\ssup{n},\nu^\ssup{n})_{n\in\N}$ be a convergent sequence in $\calM_\tem(\R)^2$ or $\calM_\rap(\R)^2$ with limit $(\mu,\nu)$.
Further, assume that for all $n\in\N$ we have 
\[R(\mu^\ssup{n})\le L(\nu^\ssup{n}).\]
Then we have also
\[R(\mu)\le\liminf_{n\to\infty} R(\mu^\ssup{n})\le\limsup_{n\to\infty} L(\nu^\ssup{n})\le L(\nu).\]
\end{corollary}

\begin{lemma}\label{lem:preservation of SST 2}
Suppose $(\mu_t^\ssup{n},\nu_t^\ssup{n})_{t\ge0}$, $n\in\N$, is a sequence of processes in $D_{[0,\infty)}(\calM_\tem(\R)^2)$ (resp.\ $D_{[0,\infty)}(\calM_\rap(\R)^2)$)
such that $(\mu_t^\ssup{n},\nu_t^\ssup{n})_{t\ge0}\to(\mu_t,\nu_t)_{t\ge0}$ weakly in $D_{[0,\infty)}(\calM_\tem(\R)^2)$ (resp.\ $D_{[0,\infty)}(\calM_\rap(\R)^2)$) w.r.t. the Meyer-Zheng topology. Suppose further that 
\[\p[R(\mu_t^\ssup{n})\le L(\nu_t^\ssup{n})\text{ for all }t>0, n\in\N]=1.\]
Then we have also
\begin{equation}\label{proof:SST4}
\p[R(\mu_t)\le L(\nu_t)\text{ for all }t>0]=1.
\end{equation}
\end{lemma}

\begin{proof}
Using \cite[Thm.\ 5]{MZ84} (see also \cite[Thm.\ 1.1(b)]{Kurtz91}), 
there exists a set $I\subseteq(0,\infty)$ of full Lebesgue measure such that by passing to a subsequence (which we do not distinguish in notation) we may assume that 
the finite dimensional distributions of $(\mu_t^\ssup{n},\nu_t^\ssup{n})_{t\in I}$ converge weakly to those of $(\mu_t,\nu_t)_{t\in I}$ as $n\to\infty$.

Fix $t\in I$. 
By Skorokhod's representation theorem we may - and will - assume that $(\mu_t^\ssup{n},\nu_t^\ssup{n})\to(\mu_t,\nu_t)$ in $\calM_\tem(\R)^2$ \emph{almost surely}. 
Applying Corollary \ref{cor:preservation of SST 1} pathwise on a set of full probability, we conclude that
\begin{equation}\label{proof:SST5}
\p[R(\mu_t)\le L(\nu_t)]=1,\qquad t\in I.
\end{equation}
Now choose a dense countable subset $J\subseteq I$ and intersect countably many sets of full probability to obtain 
\begin{equation}\label{proof:SST6}
\p[R(\mu_t)\le L(\nu_t)\text{ for all }t\in J]=1.
\end{equation}

For arbitrary $t>0$, since $J$  is dense we can choose a sequence $t_n\downarrow t$ with $t_n\in J$. Now use the fact that \eqref{proof:SST6} holds for each $t_n$ plus right-continuity of the paths of $(\mu_t,\nu_t)_{t\ge0}$ and again apply Corollary \ref{cor:preservation of SST 1} (with $(\mu^\ssup{n},\nu^\ssup{n}):=(\mu_{t_n},\nu_{t_n}))$ pathwise on a set of full probability to see that \eqref{proof:SST4} holds.
\end{proof}

We return to the discrete-space infinite rate model $\mathrm{dSBM}(\rho,\infty)_{u_0,v_0}$, with initial conditions $(u_0,v_0)\in\calM_\tem(\Z)^2$, and now discuss strict positivity of the solution $(u_t,v_t)$ for fixed positive times $t>0$. 
Let $Q^\rho_{x,y}$ denote the exit measure of a planar $\rho$-correlated Brownian motion $(B^\ssup{1},B^\ssup{2})$ from the first quadrant $(\R^+)^2$, started at $(x,y)\in(\R^+)^2$. That is, defining $\tau:=\inf\{t>0:B_t^\ssup{1} B_t^\ssup{2}=0\}$, let
\[Q_{x,y}^\rho(\cdot):=\p_{x,y}[(B_\tau^\ssup{1},B_\tau^\ssup{2})\in\cdot].\]
Then $Q_{x,y}^\rho$ is concentrated on the boundary of the first quadrant, i.e. on $(\R^+)^2\setminus(0,\infty)^2$, and has no atom at zero iff $(x,y)\ne(0,0)$.
Moreover, the mapping $(x,y)\mapsto Q_{x,y}^\rho$ is continuous. 
For an exact description of $Q^\rho_{x,y}$, see e.g.\ \cite{DM12}, in particular eq.\ (3.16) there. 
By a classical result due to \cite{KO10}, for initial conditions that are already separated (i.e.\ $u_0(k)v_0(k)=0$ for all $k\in\Z$) the exit measure $Q^\rho_{x,y}$ determines the distribution of the solution for fixed time $t>0$, as follows: We have
\begin{equation}\label{eq:fixed time distribution_discrete}
\E_{u_0,v_0}\left[Q^\rho_{\langle u_t,\,\phi\rangle_\Z,\,\langle v_t,\,\psi\rangle_\Z}(\cdot)\right]=Q^\rho_{\langle {^dS}_tu_0,\,\phi\rangle_\Z,\,\langle {^dS}_tv_0,\,\psi\rangle_\Z}(\cdot)
\end{equation}
for all suitable test functions $\phi,\psi$. 
This was derived in \cite[Thm.\ 2]{KO10} for the case $\rho=0$ as a consequence of the Trotter approximation, which can also be generalized to $\rho\in(-1,1)$ (see e.g.\ \cite{DM12}, Sec.\ 4.3, pp.\ 34ff.). 
Choosing $\phi:=\psi:=\1_{\{k\}}$ and using that obviously $Q_{x,y}^\rho(\cdot)=\delta_{(x,y)}(\cdot)$ if $xy=0$ shows that the distribution of $(u_t(k),v_t(k))$ at fixed space-time points is given by
\begin{equation}\label{eq:fixed time distribution_discrete_2}
\p_{u_0,v_0}[(u_t(k),v_t(k))\in\cdot]=Q^\rho_{{^dS}_tu_0(k),{^dS}_tv_0(k)}(\cdot),\qquad k\in\Z,\,t>0.
\end{equation}
From this it follows immediately that $u_t(k)+v_t(k)>0$ almost surely if $u_0$ and $v_0$ are not both identically zero, since then $(S_tu_0(k),S_tv_0(k))\ne(0,0)$. 

It is straightforward to generalize \eqref{eq:fixed time distribution_discrete}-\eqref{eq:fixed time distribution_discrete_2} to arbitrary (not necessarily separated) initial conditions $(u_0,v_0)\in\cM_\tem(\Z)^2$ as employed in our framework: In fact, since the separation of types holds at positive times, by the Markov property applied at time $s\in(0,t)$ we get
\begin{align}\bal
\E_{u_0,v_0}\left[Q^\rho_{\langle u_t,\,\phi\rangle_\Z,\,\langle v_t,\,\psi\rangle_\Z}(\cdot)\right]&=\E_{u_0,v_0}\left[\E_{u_s,v_s}\left[Q^\rho_{\langle u_{t-s},\,\phi\rangle_\Z,\,\langle v_{t-s},\,\psi\rangle_\Z}(\cdot)\right]\right]\\
&=\E_{u_0,v_0}\left[Q^\rho_{\langle {^dS}_{t-s}u_s,\,\phi\rangle_\Z,\,\langle {^dS}_{t-s}v_s,\,\psi\rangle_\Z}(\cdot)\right].
\eal\end{align}
Letting $s\downarrow0$ and using the right-continuity of the paths of $(u_t,v_t)_{t\ge0}$ we obtain \eqref{eq:fixed time distribution_discrete}, 
and \eqref{eq:fixed time distribution_discrete_2} follows as before.

As a first simple application of our convergence result Theorem \ref{thm:conv}, we can now easily extend \eqref{eq:fixed time distribution_discrete} to continuous space:

\begin{lemma}\label{lem:fixed time distribution}
Let $\rho\in(-1,0)$. Assume that $(\mu_0,\nu_0)\in\calM_\tem(\R)^2$ (resp.\ $\calM_\rap(\R)^2$). 
Then we have for the solution of $\mathrm{cSBM}(\rho,\infty)_{\mu_0,\nu_0}$ that
\begin{equation}\label{eq:fixed time distribution}
\E_{\mu_0,\nu_0}\left[Q^\rho_{\langle \mu_t,\,\phi\rangle_\R,\,\langle \nu_t,\,\psi\rangle_\R}(\cdot)\right]=Q^\rho_{\langle S_t\mu_0,\,\phi\rangle_\R,\,\langle S_t\nu_0,\,\psi\rangle_\R}(\cdot)
\end{equation}
for all $t\ge0$ and test functions $\phi,\psi\in\bigcup_{\lambda>0}\calC_\lambda^+(\R)$ (resp.\ $\bigcup_{\lambda>0}\calC_{-\lambda}^+(\R)$). 
\end{lemma}
\begin{proof} 
We define a sequence of approximating discrete processes as in \eqref{defn:variabel initial conditions}-\eqref{defn:rescaling}.
By \eqref{eq:rescaling} and \eqref{eq:fixed time distribution_discrete}, we have for all $n\in\N$
\[\bal
\E_{\mu_0^\ssup{n},\nu_0^\ssup{n}}\left[Q^\rho_{\left\langle \mu^\ssup{n}_{t},\,\phi\right\rangle_\R,\,\left\langle \nu^\ssup{n}_{t},\,\psi\right\rangle_\R}(\cdot)\right]
&=\E_{u_0^\ssup{n},v_0^\ssup{n}}\left[Q^\rho_{\left\langle u^\ssup{n}_{n^2t},\,\phi^\ssup{n})\right\rangle_\Z,\,\left\langle v^\ssup{n}_{n^2t},\,\psi^\ssup{n}\right\rangle_\Z}(\cdot)\right]\\
&=Q^\rho_{\left\langle S_{n^2t}u_0^\ssup{n},\,\phi^\ssup{n}\right\rangle_\Z,\,\left\langle S_{n^2t}v_0^\ssup{n},\,\psi^\ssup{n}\right\rangle_\Z}(\cdot).
\eal\]
Applying Theorem \ref{thm:conv}, Lemma \ref{lem:semigroup convergence} and using continuity of $(x,y)\mapsto Q^\rho_{x,y}$, we can take the limit on both sides to obtain \eqref{eq:fixed time distribution}.
\end{proof}

\begin{corollary}\label{cor:strict positivity}
Let $\rho\in(-1,0)$ and $(\mu_0,\nu_0)\in\calM_\tem(\R)^2$ or $\calM_\rap(\R)^2$.
Assume further that $\mu_0+\nu_0$ is not the zero measure. Then we have for 
all $t>0$ that
\[\p_{\mu_0,\nu_0}[\supp(\mu_t+\nu_t)=\R]=1.\]
\end{corollary}
\begin{proof}
Fix $t>0$ and consider a test function $\phi\in\calC_c^+$, $\phi\ne0$. Then since $\langle\mu_t+\nu_t,\phi\rangle_\R =0$ implies 
$Q^\rho_{\langle \mu_t,\,\phi\rangle_\R,\,\langle \nu_t,\,\phi\rangle_\R}(0,0)=1$, we have by Lemma \ref{lem:fixed time distribution} that
\[\bal
\p_{\mu_0,\nu_0}\left[\langle\mu_t+\nu_t,\phi\rangle_\R =0\right]&=\E_{\mu_0,\nu_0}\left[\1_{\langle\mu_t+\nu_t,\phi\rangle_\R =0}\,Q^\rho_{\langle \mu_t,\,\phi\rangle_\R,\,\langle \nu_t,\,\phi\rangle_\R}(0,0)\right]\\
&\le\E_{\mu_0,\nu_0}\left[Q^\rho_{\langle \mu_t,\,\phi\rangle_\R,\,\langle \nu_t,\,\phi\rangle_\R}(0,0)\right]=Q^\rho_{\langle S_t\mu_0,\,\phi\rangle_\R,\,\langle S_t\nu_0,\,\phi\rangle_\R}(0,0)=0.
\eal\]
Here we used that $Q^\rho_{x,y}$ has an atom at zero iff $(x,y)=(0,0)$.
Since $\phi$ was arbitrary, this implies in particular that for all intervals $[a,b]$ with rational endpoints $a,b\in\Q$ we have $\p_{\mu_0,\nu_0}[(\mu_t+\nu_t)[a,b]>0]=1$. Intersecting  the countably many sets of probability one, we see that $\supp(\mu_t+\nu_t)=\R$ almost surely.
\end{proof}

Before turning to the proof of Theorem \ref{thm:SST}, we now sketch a proof of an analogous version for the discrete-space case. 
Again let $(u_t,v_t)_{t\ge0}$ denote the solution to $\mathrm{dSBM}(\rho,\infty)_{u_0,v_0}$, with $(u_0,v_0)\in\calM_\tem(\Z)^2$. 
We consider $(u_t,v_t)$ as elements of $\calM_\tem(\R)^2$ which are concentrated on the lattice $\Z\subseteq\R$. 
The support is given by $\supp(u_t)=\{k\in\Z:u_t(k)>0\}$, and $R(u_t)$, $L(v_t)$ are defined as before. 

\begin{prop}\label{prop:SST_SIP_discrete}
Let $\rho\in(-1,1)$ and consider initial conditions $(u_0,v_0)\in\calM_\tem(\Z)^2$ or $\cM_\rap(\Z)^2$
such that $R(u_0)<L(v_0)$.
Then almost surely the initial ordering of types is preserved for all times, i.e.
\begin{equation}\label{eq:discrete_perservation} \p_{u_0,v_0}[R(u_t)<L(v_t)\text{ for all }t>0]=1. \end{equation}
Assume in addition that $u_0+v_0$ is not identically zero. Then we have the discrete analogue of the `single-point interface' property for all fixed times, in the sense that
\[\p_{u_0,v_0}[R(u_t)=L(v_t)-1]=1,\qquad t>0.\]
\end{prop}
\begin{proof}
We only sketch the argument: Recall the approximation procedure used in \cite{KM11b} in order to construct the discrete model $\mathrm{dSBM}(\rho,\infty)_{u_0,v_0}$ for $\rho=0$ (see \cite{DM12} for the extension to $\rho\ne0$). Inspection of the definition in \cite{KM11b}, p. 15, eqs. (2.8)-(2.9), shows that the
initial ordering of types is preserved by the dynamics of the approximating processes, thus~\eqref{eq:discrete_perservation} holds for the latter. 
By \cite[Thm.\ 3.1]{KM11b}, the approximating processes converge to $(u_t,v_t)_{t\ge0}$ w.r.t.\ the Skorokhod topology.
Applying Lemma \ref{lem:preservation of SST 2}, 
we infer that the limit satisfies $\p_{u_0,v_0}[R(u_t)\le L(v_t)\text{ for all }t>0]=1$. 
However, the inequality is in fact strict because we know already by the usual separation-of-types property \eqref{eq:dsep} of $\mathrm{dSBM}(\rho,\infty)_{u_0,v_0}$ that almost surely $u_t(k)v_t(k)=0$ for all $k\in\Z$, $t>0$. 

In addition, if $u_0+v_0$ does not vanish, then for all $t>0$, almost surely $u_t+v_t$ is strictly positive by \eqref{eq:fixed time distribution_discrete_2}.
But $R(u_t)< L(v_t)$ together with $u_t+v_t>0$ implies $R(u_t)=L(v_t)-1$.
\end{proof}

\begin{proof}[Proof of Theorem \ref{thm:SST}]
Consider $\rho\in(-1,0)$ and 
initial conditions $(\mu_0,\nu_0)\in\cM_\tem(\R)^2$ which are mutually singular and such that $R(\mu_0)\le L(\nu_0)$. 

If $R(\mu_0)<L(\nu_0)$, then for the initial conditions \eqref{defn:variabel initial conditions} for the discrete process we also have $R(u_0^\ssup{n})<L(v_0^\ssup{n})$, provided $n\in\N$ is large enough.
Now consider the case $R(\mu_0)= L(\nu_0)=:I_0$, where we can assume without loss of generality that $I_0=0$. 
If $\mu_0(\{0\})=0$, then \eqref{defn:variabel initial conditions} implies $R(u_0^\ssup{n})=-1<0=L(v^\ssup{n}_0)$.
If on the other hand $\mu_0(\{0\})>0$, then we must have $\nu_0(\{0\})=0$ since by assumption the measures $\mu_0$ and $\nu_0$ are mutually singular and thus cannot both have an atom at $I_0=0$. 
Modifying the definition of the discrete initial conditions according to \eqref{defn:variabel initial conditions 2}, we then again have $R(u_0^\ssup{n})=-1<0=L(v^\ssup{n}_0)$.
Thus in both cases Proposition~\ref{prop:SST_SIP_discrete} is applicable, and using the definition \eqref{defn:rescaling} of the approximating processes we get that almost surely
\[R(\mu_t^\ssup{n})\le L(\nu_t^\ssup{n})\]
for all $t>0$ and $n\in\N$. 
Another application of Lemma \ref{lem:preservation of SST 2} in combination with Theorem~\ref{thm:conv} concludes the proof of the `preservation of order of types'-property \eqref{proof:SST1}.

For the rest of the proof, we fix $t>0$. 
As for the discrete model, the `single-point interface' property \eqref{proof:SST2} is now a simple consequence of \eqref{proof:SST1} together with strict positivity from Corollary \ref{cor:strict positivity}.
We have to show that $\mu_t$ and $\nu_t$ are mutually singular almost surely.
Since $I_t:=R(\mu_t)=L(\nu_t)$ a.s., it is clear by the definition of the measure-theoretic support that $\mu_t((I_t,\infty))=0$ and $\nu_t((-\infty,I_t))=0$ a.s.
It only remains to rule out the possibility that both $\mu_t$ and $\nu_t$ have an atom at $I_t$ with positive probability. Choose any 
strictly positive test function 
$\phi\in\bigcup_{\lambda>0}\cC_\lambda^+(\R)$, $\phi(\cdot)>0$. Then we have
\begin{align*}
&\int_\R \phi(x)\,\E_{\mu_0,\nu_0}\left[S_\eps\mu_t(x) S_\eps\nu_t(x)\right]dx\\
&\quad=\int_\R \phi(x)\,\E_{\mu_0,\nu_0}\left[\int_\R p_\eps(x-y)\mu_t(dy)\int_\R p_\eps(x-z)\nu_t(dz)\right]dx\\
&\quad\ge\int_\R \phi(x)\,\E_{\mu_0,\nu_0}\left[p_\eps(x-I_t)\mu_t(\{I_t\})\,p_\eps(x-I_t)\nu_t(\{I_t\})\right]dx\\
&\quad=\E_{\mu_0,\nu_0}\left[\mu_t(\{I_t\})\nu_t(\{I_t\})\int_\R \phi(x)\, p_\eps(x-I_t)^2\,dx\right]\\
&\quad=\frac{1}{2\sqrt{\pi\eps}}\,\E_{\mu_0,\nu_0}\left[\mu_t(\{I_t\})\nu_t(\{I_t\}) \int_\R \phi(x/\sqrt{2})\,p_\eps(x-\sqrt{2}I_t)\,dx\right].
\end{align*}
In the above display, the LHS goes to zero by the separation-of-types property \eqref{eq:csep}, while the integral inside the expectation on the RHS converges to $\phi(I_t)>0$ a.s. This leads to a contradiction unless we have $\mu_t(\{I_t\})\nu_t(\{I_t\})=0$ a.s. Hence $\mu_t$ and $\nu_t$ cannot both have an atom at $I_t$, and our proof is complete.
\end{proof}

\appendix
\section{Appendix}
\subsection{Notation and spaces of functions and measures}\label{appendix0}

In this appendix  we have collected our notation and we recall some well-known facts concerning the spaces of functions and measures employed throughout the paper. Most of the material in this subsection can also be found e.g. in \cite{Dawsonetal2002}, \cite{Dawsonetal2003} or \cite{EF04}.
We can develop the notation for both the discrete and the continuous setting simultaneously,
so throughout we let $\cS=\Z^d$ for some $d\in\N$ or $\cS=\R$.

For $\lambda \in \mathbb{R}$, let 
\[\phi_{\lambda}(x) := e^{- \lambda |x|},\quad  x \in \cS,\] 
and for $f: \cS\rightarrow \mathbb{R}$ define
\[|f|_{\lambda} := ||f / \phi_{\lambda}||_{\infty},\] 
where $||\cdot||_{\infty}$ is the supremum norm. 
We denote by $\mathcal{B}_{\lambda}(\cS)$ the space of all measurable
functions $f: \cS \rightarrow \mathbb{R}$ such that $|f|_{\lambda}
<\infty$ and so that $f(x) / \phi_{\lambda}(x)$ has a finite limit
as $|x| \rightarrow \infty$. Next, we introduce the spaces of 
\textit{rapidly decreasing} and \textit{tempered} measurable 
functions, 
respectively, as
\begin{equation}
\mathcal{B}_{\rap}(\cS) := 
 \bigcap_{\lambda > 0} \mathcal{B}_{\lambda}(\cS) \quad \textrm{ and } \quad
\mathcal{B}_{\tem} :=
\bigcap_{\lambda > 0} \mathcal{B}_{- \lambda}(\cS) . 
\end{equation}

For $\cS = \R$, we write $\mathcal{C}_{\lambda}(\R), \mathcal{C}_{\rap}(\R), \mathcal{C}_{\tem}(\R)$ 
for the subspaces of continuous functions in $\calB_\lambda(\R), \calB_\rap(\R),\calB_\tem(\R)$ respectively. 
Moreover, if we additionally assume that all partial derivatives up to order $k\in\N$ exist and belong to $\calC_\lambda(\R),\calC_\rap(\R),\calC_\tem(\R)$, we write 
$\calC_\lambda^{(k)}(\R),\mathcal{C}_{\rap}^{(k)}(\R), \mathcal{C}_{\tem}^{(k)}(\R)$.
In order to formulate results for the continuous and discrete case simultaneously, it is convenient to employ all of these notations also for $\cS=\Z^d$,
where of course (by convention) the spaces just introduced coincide with $\calB_\lambda(\Z^d), \calB_\rap(\Z^d),\calB_\tem(\Z^d)$ respectively.
Finally, we will also use the space $\calC_c^\infty(\R)$ of infinitely differentiable functions with compact support. 

For each $\lambda \in \mathbb{R}$, the linear space $\mathcal{C}_{\lambda}(\cS)$ endowed 
with the norm $| \cdot |_{\lambda}$ is a separable Banach space, and 
the spaces $\mathcal{C}_{\rap}(\cS), \mathcal{C}_{\tem}(\cS)$ can be topologized by a suitable metric 
to turn them into Polish spaces, for the details see e.g.\ Appendix A.1 in~\cite{BHO15}.

If $\calF$ is any of the above spaces of functions, we denote by $\calF^+$  the subset of nonnegative elements of $\calF$.

Let $\mathcal{M}(\cS) $ denote the space of
(nonnegative) Radon 
measures on $\cS$.
For $\mu\in\calM(\cS)$ and a measurable function $f:\cS\to\R$, we will 
 denote the integral of $f$ with
respect to the measure $\mu$ (if it exists)
by any of the following notations
\[\langle \mu , f \rangle,\quad \int_\cS\mu(dx)\,f(x),\quad \int_\cS f(x)\,\mu(dx). \]
In the case of the Lebesgue measure $\ell$ on $\R$, 
we will simply write $dx$ in place of $\ell(dx)$. If $\mu\in\calM(\R)$ is absolutely continuous w.r.t. $\ell$, we will identify $\mu$ with its density, writing
\[\mu(dx)=\mu(x)\,dx.\]
Similarly, for $\mu \in \calM(\Z)$, we will often write $\mu(k) := \mu(\{ k\})$.

For $\lambda\in\R$, define 
\[\calM_\lambda(\cS):=\left\{\mu\in\calM(\cS):\langle\mu,\phi_\lambda\rangle < \infty\right\}\]
and introduce the spaces
\[\mathcal{M}_{\tem}(\cS) := \bigcap_{\lambda > 0}\calM_\lambda(\cS),\qquad \calM_\rap(\cS):=\bigcap_{\lambda>0}\calM_{-\lambda}(\cS) \]
of \textit{tempered} and \textit{rapidly decreasing measures} on $\cS$, respectively.
Again by defining suitable metrics it can be seen that these spaces are Polish.
Moreover,  $\mu_n\to\mu$ in $\calM_\tem(\cS)$
iff $\langle\mu_n,\varphi\rangle\to\langle\mu,\varphi\rangle$ for all $\varphi\in\bigcup_{\lambda > 0} \calC_{\lambda}(\cS)$.
We write $\mathcal{M}_f(\cS)$ for
the space of finite measures on
$\cS$ endowed with the topology of weak convergence. With this notation we have $\calM_\rap(\cS)\subseteq\calM_f(\cS)$. 
To topologize the space $\calM_\rap(\cS)$ we say that $\mu_n\to\mu$ in $\calM_\rap(\cS)$ iff $\mu_n\to\mu$ in $\calM_f(\cS)$ (w.r.t.\ the weak topology) and $\sup_{n\in\N}\langle\mu_n,\phi_\lambda\rangle<\infty$ for all $\lambda<0$ (see \cite{Dawsonetal2003}, p. 140). 

Finally, we remark that 
$\cB^+_\tem(\R)$ and thus also $\calC^+_\tem(\R)$ may be viewed as subspaces of $\calM_\tem(\R)$: Indeed, we can take a function $u\in\calB_\tem^+(\R)$ 
as a density w.r.t.\ Lebesgue measure and thus identify it with the tempered measure $u(x)\, dx$. 
Note, however, that the topology of $\calM_\tem(\R)$ restricted to $\calC_\tem^+(\R)$ is weaker than the topology on $\calC_\tem(\R)$ defined in Appendix A.1 in~\cite{BHO15}. 
The same relationship holds  between $\calC_\rap^+(\R)$ and $\calM_\rap(\R)$. 
In particular, we have \emph{continuous} embeddings $\calC^+_\tem(\R)\hookrightarrow\calM_\tem(\R)$ and $\calC_\rap^+(\R)\hookrightarrow\calM_\rap(\R)$. 

For $\cS=\Z^d$, we can identify any measure in $\cM_\tem(\Z^d)$ resp.\ $\cM_\rap(\Z^d)$ with its density w.r.t.\ counting measure, and moreover this density will be in $\cB_\tem^+(\Z^d)$ resp.\ $\cB_\rap^+(\Z^d)$, the reason being that a summable sequence is necessarily bounded. Thus we have equality of the \emph{sets} $\cM_\tem(\Z^d)=\cB_\tem^+(\Z^d)$ and $\cM_\rap(\Z^d)=\cB_\rap^+(\Z^d)$. Nevertheless, as in the continuous case, the topology on $\cM_\tem(\Z^d)$ resp.\ $\cM_\rap(\Z^d)$ introduced above is strictly weaker than the topology on $\cB_\tem^+(\Z^d)$ resp.\ $\cB_\rap^+(\Z^d)$ defined in \cite[Appendix A.1]{BHO15}.

\subsection{Semigroup estimates}\label{sec:semigroup}

Let $(p_t)_{t\ge0}$ denote the heat
kernel in $\mathbb{R}$ corresponding to $\frac{1}{2} \Delta$,
\begin{equation}
p_t(x)  = \frac{1}{(2 \pi  t)^{1/2}} \exp \left\{- \frac{|x|^2}{2  t} \right\}, 
\qquad t > 0, x \in \mathbb{R},
\end{equation}
and write $ (S_t)_{ t \geq 0}$ for the associated heat semigroup. 
Similarly, let $(^dS_t)_{t\ge0}$ denote the semigroup corresponding to a continuous-time simple random walk $(X_t)_{t \geq 0}$
with generator $\frac 12 {^d\Delta}$, the discrete Laplace operator as defined in~\eqref{def:discreteLaplace}, and discrete heat kernel ${^dp}_t$.

For $\mu\in\calM(\R)$ and $x\in\R$, let $S_t\mu(x):=\int_\R p_t(x-y)\,\mu(dy)$ and similarly for $^dS$. 
The following estimates are well known and can be proved as in Appendix A of \cite{Dawsonetal2003} (see also \cite[Lemma 6.2 (ii)]{Shiga94}):
\begin{lemma}\label{lemma estimates}
Fix $\lambda\in\R$ and $T>0$. 
\begin{itemize}
\item[a)] 
For all $\varphi\in\calB_{\lambda}^+(\R)$, we have
\begin{equation}\label{estimate 1}
\sup_{t\in[0,T]}S_t\varphi(x)\le C(\lambda,T)\,|\varphi|_\lambda\, \phi_\lambda(x),\qquad x\in\R.
\end{equation}
Moreover, there is a constant $C'(\lambda,T)>0$ such that
\begin{equation}\label{estimate 1a}
\inf_{t\in[0,T]}S_t\phi_{\lambda}(x)\ge C'(\lambda,T)\,\phi_{\lambda}(x) ,\qquad x\in\R.
\end{equation}
\item[b)] 
Let $0<\eps<T$. Then for
all $\mu\in\calM_\lambda(\R)$ we have 
\begin{equation}\label{estimate 2}
 \sup_{t\in[\eps,T]}S_t\mu(x)
 \le C(\lambda,T,\eps)\,\langle\mu,\phi_{\lambda}\rangle\,\phi_{-\lambda}(x),\qquad x\in\R.
\end{equation}
\end{itemize}
Therefore, the heat semigroup preserves the space $\calB_\lambda(\R)$ and maps $\calM_\lambda(\R)$ into $\calB_\lambda(\R)$.
\end{lemma}

We have analogous estimates for the discrete space semigroup:
\begin{lemma}\label{lem:uniform preservation}
Let $\lambda\in\R$ and $T>0$. Then there are constants $c(\lambda,T), C(\lambda,T) > 0$ 
such that for all $n\in\N$, $k\in\Z$ and all $t \in [0,T]$
we have
\begin{equation}\label{eq:uniform preservation}
c(\lambda, T) \, \phi_{\lambda}(\tfrac{k}{n}) \leq {^dS}_{n^2t}(\phi_{\lambda}(\tfrac{\cdot}{n}))(k)\le C(\lambda, T)\,\phi_{\lambda}(\tfrac{k}{n}).
\end{equation}
\end{lemma}
\begin{proof} 
The upper bound is just a reformulation of \cite[Corollary A3(a)]{Dawsonetal2002}.
For the lower bound, let $X_t$ be a continuous time random walk with generator $\frac{1}{2}{^d\Delta}$
started in $0$. 
If we define $X_t^\ssup{n} := X_{n^2 t}/n$, then we know from Donsker's theorem that 
$(X_t^\ssup{n})_{t \geq 0}$ converges in distribution to a standard Brownian motion $B$.
Fix $T > 0$, by Skorokhod's representation theorem, we can choose a common probability space
$\p$ (with expectation $\E$)
such that $\sup_{t \in [0,T]} |X_t^\ssup{n} - B_t| \ra 0$ almost surely.

For $\la \geq 0$, we can estimate using the triangle inequality
\[ ^d S_{n^2 t} (\phi_\la(\cdot / n)) (k) = 
\E[ e^{-\la |X^\ssup{n}_t + k | } ] 
\geq e^{-\la |k/n|} \E [ e^{-\la |X_t^\ssup{n}|} ] . \]
So it remains to show that the expectation on the LHS is bounded from below uniformly in $t\in [0,T]$ and $n$.
Now, choose $n_0$ large enough such that for all $n \geq n_0$ 
\[ \p \Big\{ \sup_{t \in [0,T]} |X_t^\ssup{n} - B_t | \geq \tfrac 12 \Big\} \leq  \tfrac 12 \E [ e^{-\la |B_T|} ] . \]
Using the above estimate and the fact that $t \mapsto \E [e^{-\la |B_t|}]$ is decreasing, we thus obtain
for $n \geq n_0$
\[\begin{aligned} \E \big[e^{- \la |X^\ssup{n}| } \big]
 & \geq e^{- \frac 12\la } \E [ e^{-\la |B_t| }\1_{\{  \sup_{t \in [0,T]} |X_t^\ssup{n} - B_t | \leq \frac 12 \}} ] \\
& \geq e^{-\frac 12\la} \Big( \E [ e^{-\la |B_T| }]  - \p \Big\{ \sup_{t \in [0,T]} |X_t^\ssup{n} - B_t | \geq \tfrac 12 \Big\} \Big) 
\geq \tfrac 12  e^{-\frac 12\la} \E [ e^{-\la |B_T| }].
\end{aligned}\]
This proves the claim for $\la \geq 0$, since for any $n \leq n_0$ we can use the trivial estimate
\[ \E [ e^{-\la |X_t^\ssup{n}|} ] \geq \p \{ X^\ssup{n}_s = 0 \mbox{ for all } s \in [0,t] \} 
\geq \p \{ X^\ssup{n}_s = 0 \mbox{ for all } s \in [0,T] \} .
\, \]
so that we choose the  constant $c(\la, T)$ as claimed.

Finally, if $\la < 0$, we can  use that
\[ ^d S_{n^2 t} (\phi_\la (\cdot / n)) (k)  = \E[ e^{-\la |X_{n^2t} + k| /n } ] \geq  e^{-\la |k/n|}  \E [ e^{\la |X^\ssup{n}_t|}] , \]
and the latter expectation can be bounded uniformly in $n$ and $t \in [0,T]$ as in case $\la \geq 0$.
\end{proof}

We also need the following local central limit theorem, which is just a reformulation of Lemma 8 and Lemma 59 in \cite{Dawsonetal2002a}:
\begin{lemma}\label{lem:local CLT}
Let $p_t$ resp.\ $p_t^\ssup{2}$ denote the usual one- resp.\ two-dimensional heat kernel, and ${^dp}_t$ resp.\ ${^dp}_t^\ssup{2}$ its discrete counterpart. 
Then we have for all $t>0$ 
\begin{align}\bal\label{local CLT}
\lim_{n\to\infty}\sup_{x\in\Z}\left|n\,{^dp}_{n^2t}(x) - p_t(x/n)\right| = 0,\\
\lim_{n\to\infty}\sup_{x\in\Z^2}\left|n^2\,{^dp}^\ssup{2}_{n^2t}(x) - p_t^\ssup{2}(x/n)\right| = 0.
\eal\end{align}
\end{lemma}

For the next lemma, we recall that $(\tilde S_t)_{t\ge0}$ (resp.\ $({^d\tilde S}_t)_{t\ge0})$ denotes the semigroup of two-dimensional standard Brownian motion (resp. simple symmetric random walk) killed upon hitting the diagonal in $\R^2$ (resp. $\Z^2$), as defined in \eqref{eq: tilde S discrete}.

\begin{lemma}[Convergence of semigroups]\label{lem:semigroup convergence}
Suppose $( \mu_0, \nu_0)\in\calM_\tem(\R)^2$ (resp.\ $\calM_\rap(\R)^2$) and $\phi,\psi\in\bigcup_{\lambda>0}\calC_{\lambda}^+(\R)$ (resp.\ $\bigcup_{\lambda>0}\calC_{-\lambda}^+(\R)$). For each $n\in\N$, define $(u_0^\ssup{n},v_0^\ssup{n})$ by \eqref{defn:variabel initial conditions} and $(\phi^\ssup{n},\psi^\ssup{n})$ by \eqref{eq:phi^n}.
Then we have the convergence
\begin{equation}\label{eq:conv_semigroup}
\left\langle\phi^\ssup{n}, {^d S}_{n^2t}u_0^\ssup{n}\right\rangle_{\Z}\xrightarrow{n\to\infty}\left\langle\phi, S_t\mu_0\right\rangle_{\R}
\end{equation}
and
\begin{equation}\label{eq:conv_killed_semigroup}
\left\langle\phi^\ssup{n}\otimes\psi^\ssup{n}, {^d\tilde S}_{n^2t}(u_0^\ssup{n}\otimes v_0^\ssup{n})\right\rangle_{\Z^2}\xrightarrow{n\to\infty}\left\langle\phi\otimes\psi,\tilde S_t(\mu_0\otimes\nu_0)\right\rangle_{\R^2}
\end{equation}
for all $t>0$.
\end{lemma}
Given the local central limit theorem in Lemma \ref{lem:local CLT}, \eqref{eq:conv_semigroup}-\eqref{eq:conv_killed_semigroup} follow by standard arguments, e.g.\ along the lines of the proof of \cite[Lemma 50]{Dawsonetal2002}. 
For the killed semigroup, one uses that the transition density $\tilde p_t$ of $\tilde S_t$ is given by
\begin{align}\label{transition density tilde S}\bal
{\tilde p}_t(x,y;a,b) =&\begin{cases}
\1_{\{a<b\}} \big(p^\ssup{2}_t(x-a,y-b) - p^\ssup{2}_t(x-b,y-a)\big) &\text{if }x<y \\
\1_{\{a>b\}} \big(p^\ssup{2}_t(x-a,y-b) - p^\ssup{2}_t(x-b,y-a)\big) &\text{if }x>y \\
\end{cases}\\
& = \left(\1_{\{x<y, a<b\}} + \1_{\{x>y, a>b\}}\right)\big(p^\ssup{2}_t(x-a,y-b) - p^\ssup{2}_t(x-b,y-a)\big) ,
\eal\end{align}
where $p_t^\ssup{2}$ denotes the usual two-dimensional heat kernel.
The corresponding discrete-space transition density reads
\begin{align}\label{transition density tilde S_discrete}\bal
{^d\tilde p}_t(k,\ell;a,b) 
& = \left(\1_{\{k<\ell, a<b\}} + \1_{\{k>\ell, a>b\}}\right)\big({^dp}_t^\ssup{2}(k-a,\ell-b) - {^dp}_t^\ssup{2}(k-b,\ell-a)\big) .
\eal\end{align}
In particular, by the above form of the densities (and the symmetry of the usual heat kernel) it is immediately seen that these semigroups are symmetric.

\subsection{The topology on path space}\label{sec:appendix_path}

Suppose  $E$ is a Polish space and $I\subseteq\R$, then we denote 
by $D_I(E)$ resp.\ $\calC_I(E)$ the space of c\`adl\`ag resp.\ continuous $E$-valued paths $t\mapsto f_t$, $t\in I$.
In this paper, we will always have $I=[0,\infty)$ or $I=(0,\infty)$ and $E\in\{(\calC^+_\tem)^m,(\calC^+_\rap)^m,\calM_\tem(\cS)^m,\calM_\rap(\cS)^m\}$ for $\cS$ either $\Z$ or $\R$ and some power $m\in\N$. 
The space $D_{I}(E)$ is then also Polish is we
endow it with the usual Skorokhod ($J_1$)-topology. 

In this paper, we will also make use of
 the weaker \emph{Meyer-Zheng `pseudo-path' topology} on $D_{[0,\infty)}(E)$. 
This topology was introduced in \cite{MZ84} and can be formalized as follows:
 let $\lambda (dt):= \exp(-t)\,  dt$ and let $w(t), t \in [0,\infty)$ be an $E$-valued Borel function. Then, a `pseudo-path' corresponding to $w$ is defined to be the probability law $\psi_w$ on $[0, \infty) \times E$ given as the image measure of $\lambda$ under
the mapping $ t \mapsto (t, w(t))$. Note that with this definition two functions which are equal Lebesgue-a.e.\ give rise to the same pseudo-path. Moreover, $w \mapsto \psi_w$ is one-to-one on the space of c\`adl\`ag paths $D_{[0,\infty)}(E)$, and thus we obtain an embedding of $D_{[0,\infty)}(E)$ into the space of probability measures on $[0, \infty) \times E$. 
The induced topology on $D_{[0,\infty)}$ is then called the pseudo-path topology. Note that convergence in this topology is equivalent to convergence in Lebesgue measure (see \cite[Lemma~1]{MZ84}). 

We will need the following sufficient condition  for relative compactness of a sequence of stochastic processes on $D_{[0,\infty)}(E)$ equipped with this topology, due to \cite[Thm.\ 4]{MZ84} in the case $E=\R$.
If $(X^\ssup{n}_t)_{t\ge0}$, $n\in\N$ is a sequence of c\`adl\`ag real-valued stochastic processes, with $(X^\ssup{n}_t)_{t\ge0}$ adapted to a filtration $(\calF^\ssup{n})_{t\ge0}$, then Meyer and Zheng require that
\begin{equation}\label{MZ condition}
\sup_{n\in\N}\Big(V_T(X^\ssup{n})+\sup_{t\le T}\E[|X_t^\ssup{n}|]\Big)<\infty
\end{equation}
for all $T>0$, where $V_T(X^\ssup{n})$ denotes the conditional variation of $X^\ssup{n}$ up to time $T$, 
defined as 
\[V_T(X^\ssup{n}) :=\sup\E\Big[\sum_i\Big|\E[X^\ssup{n}_{t_{i+1}}-X_{t_i}^\ssup{n}\,|\,\calF^\ssup{n}_{t_i}]\Big|\Big], \] and the $\sup$ is taken over all partitions of the interval $[0,T]$.
For our purposes we need the version  for 
processes taking values in general separable metric spaces $E$
stated in \cite{Kurtz91}. 
In fact, according \cite[Cor.\ 1.4]{Kurtz91} we only have to check condition \eqref{MZ condition} for the coordinate processes and in addition a compact containment condition to deduce tightness of our measure-valued processes in the pseudopath topology (which again is equivalent to the topology of convergence in Lebesgue measure).

%%%%%%%%%%%%%%%%%%%%%%%%%%%%%%%%%%%%%%%%%%%%%%%%%%%%%%%%%%%%%%%%%%%%%%%%%%%%%%
%	Bibliography 							     %
%%%%%%%%%%%%%%%%%%%%%%%%%%%%%%%%%%%%%%%%%%%%%%%%%%%%%%%%%%%%%%%%%%%%%%%%%%%%%%
%\newpage
%\setlength{\bibsep}{0.1\baselineskip}
\bibliographystyle{alpha}
\bibliography{duality}

\end{document}